\documentclass{amsart}
\usepackage{breqn}
\usepackage{schock}

\title{Quasilinear tropical compactifications}

\author{Nolan Schock}
\email{nschoc2@uic.edu}
\address{Department of Mathematics, University of Illinois at Chicago, Chicago
IL, 60607, USA}

\begin{document}
\maketitle

\begin{abstract}
  The prototypical examples of tropical compactifications are compactifications of complements of hyperplane
  arrangements, which posses a number of remarkable properties not satisfied by more general tropical compactifications
  of closed subvarieties of tori. We introduce a broader class of tropical compactifications, which we call
  \emph{quasilinear (tropical) compactifications}, and which continue to satisfy the desirable properties of
  compactifications of complements of hyperplane arrangements. In particular, we show any quasilinear compactification
  is sch\"on, and its intersection theory is described entirely by the intersection theory of the corresponding tropical
  fan. As applications, we prove the quasilinearity of the moduli spaces of 6 lines in $\bP^2$ and marked cubic
  surfaces, obtaining results on the geometry of the stable pair compactifications of these spaces.
\end{abstract}

\section{Introduction}

Let $Y$ be a $d$-dimensional closed subvariety of an $n$-dimensional torus $T \cong (\bC^*)^n$. It is natural to attempt
to compactify $Y \subset T$ by taking its closure $\oY$ in a toric variety $X(\Sigma)$ with torus $T$. If $\oY$ is
proper and the induced multiplication map $T \times \oY \to X(\Sigma)$ is flat and surjective, then $\oY$ is called a
\emph{tropical compactification} of $Y$ \cite{tevelevCompactificationsSubvarietiesTori2007}. The definition implies that
if $\oY \subset X(\Sigma)$ is a tropical compactification, then it intersects each torus orbit in $X(\Sigma)$
nontrivially in the expected dimension; therefore, one can pullback the stratification of $X(\Sigma)$ by torus orbits to
a stratification of $\oY$, giving one a method to compare the geometry of $\oY$ to the geometry of the ambient toric
variety $X(\Sigma)$.

The goal of the present article is to study how close the geometry of a tropical compactification is to the geometry of
the ambient toric variety. Tropical compactifications can be quite general (see \cite[Chapter
6]{maclaganIntroductionTropicalGeometry2015} for many interesting examples), so at first glance one should not expect to
be able to say much. On the other hand, the prototypical example of a tropical compactification is when $Y$ is a
complement of a hyperplane arrangement, in which case tropical compactifications of $Y$ satisfy a number of remarkable
properties:
\begin{itemize}
  \item $Y$ is sch\"on, i.e., there is a tropical compactification $\oY \subset X(\Sigma)$ of $Y$ such that the
      multiplication map $T \times \oY \to X(\Sigma)$ is \emph{smooth} and surjective. This implies that \emph{any}
      tropical compactification of $Y$ has smooth multiplication map \cite[Theorem
      1.4]{tevelevCompactificationsSubvarietiesTori2007}, and if $\oY \subset X(\Sigma)$ is any such tropical
      compactification, then all strata of $\oY$ are smooth \cite[Lemma 2.7]{hackingHomologyTropicalVarieties2008}. In
      fact, when $\oY \subset X(\Sigma)$ is a complement of a hyperplane arrangement, all strata of $\oY$ are also
      complements of hyperplane arrangements \cite{katzRealizationSpacesTropical2011}.
  \item If $\oY \subset X(\Sigma)$ is a tropical compactification of $Y$, then there is an isomorphism of Chow rings
    \[
        A^*(\oY) \cong A^*(X(\Sigma)).
    \]
    If in addition the toric variety $X(\Sigma)$ is nonsingular, then $\oY$ is also nonsingular, and the Chow ring of
    $\oY$ agrees with its cohomology ring \cite{grossIntersectionTheoryLinear2015, ardilaLagrangianGeometryMatroids2022,
    adiprasitoHodgeTheoryCombinatorial2018}.
\end{itemize}
We call complements of hyperplane arrangements \emph{linear varieties}, and their tropical compactifications
\emph{linear (tropical) compactifications}.

We define a broader class of closed subvarieties of tori, called \emph{quasilinear} varieties, and their tropical
compactifications, called \emph{quasilinear (tropical) compactifications}, which includes the class of linear varieties
(and linear tropical compactifications), for which the analogues of the above properties hold. Our main results are the
following.

\begin{theorem}[{\cref{thm:qlin_vars,thm:qlin_strata,cor:qlin_schon}}] \label{thm:qlin_schon_main}
    \begin{enumerate}
        \item Quasilinear varieties are smooth, irreducible, rational, Chow-free, and linearly stratified (see
            \cref{def:chow_free,def:weakly_linear}).
        \item Every stratum of a quasilinear tropical compactification is a quasilinear variety.
    \end{enumerate}
    In particular, quasilinear varieties are sch\"on.
\end{theorem}

\begin{theorem}[{\cref{thm:qlin_chow}}] \label{thm:qlin_chow_main}
    Let $i : \oY \hookrightarrow X(\Sigma)$ be a quasilinear tropical compactification. Then for all $k$, the pullback
    $i^* : A^k(X(\Sigma)) \to A^k(\oY)$ is an isomorphism, inducing an isomorphism of Chow rings
    \[
        A^*(\oY) \cong A^*(X(\Sigma)).
    \]
    If in addition the toric variety $X(\Sigma)$ is nonsingular, then $\oY$ is also nonsingular, and the cycle class map
    induces an isomorphism $A^*(\oY) \cong H^*(\oY)$.
\end{theorem}

\begin{remark}
    We emphasize in particular that the above theorem yields an isomorphism of Chow rings even in the case where
    $\Sigma$ is not simplicial. This is a generalization of previous results even in the linear case.
\end{remark}

The definition of quasilinearity is inspired by recent work of Amini and Piquerez on Hodge theory for tropical fans
\cite{aminiHodgeTheoryTropical2020, aminiHomologyTropicalFans2021}. Combinatorial Hodge theory, which has recently been
developed to resolve a number of long-standing conjectures in combinatorics (cf.
\cite{adiprasitoHodgeTheoryCombinatorial2018, ardilaLagrangianGeometryMatroids2022,
huhCombinatorialApplicationsHodgeRiemann2019}), begins from the observation that the isomorphism of cohomology rings
$H^*(\oY) \cong A^*(X(\Sigma))$ for linear tropical compactifications endows an \textit{a priori} purely combinatorial
object (the Chow ring of a realizable matroid) with the rich structure of the cohomology ring of a smooth projective
algebraic variety. In particular, this ring satisfies Poincar\'e duality, Hard Lefschetz, and the Hodge-Riemann
relations (the so-called ``K\"ahler package''), which implies a number of remarkable combinatorial consequences. It was
soon observed that even Chow rings of non-realizable matroids satisfy the K\"ahler package
\cite{adiprasitoHodgeTheoryCombinatorial2018}, and inspired by this Amini and Piquerez have defined a broader class of
\emph{shellable} tropical fans, continuing to satisfy the K\"ahler package \cite{aminiHomologyTropicalFans2021}.

We define in a slightly different fashion the notion of a \emph{quasilinear} tropical fan, which it turns out is the
same as a shellable tropical fan in the sense of Amini and Piquerez (\cref{thm:qlin_shellable}). Namely, we say a
reduced tropical fan is \emph{quasilinear} if it is isomorphic to a complete tropical fan, or a fan supported on a
tropical modification of a quasilinear tropical fan along a quasilinear tropical divisor (\cref{def:qlin_fans}). (A
reduced tropical fan is one whose weights are all 1. Tropical modifications, introduced by Mikhalkin
\cite{mikhalkinTropicalGeometryIts2007}, are the tropical versions of graphs of piecewise linear polynomials---the
classical graph is unbalanced, so one must modify it by adding in some additional cones. For precise definitions, see
\cref{sec:trop_prelims,sec:qlin_fans}). We say a closed subvariety $Y$ of an algebraic torus $T$ is \emph{quasilinear}
if its tropicalization is (the support of) a quasilinear tropical fan, and we define a \emph{quasilinear tropical
compactification} to be any tropical compactification of a quasilinear algebraic variety.  This inductive definition of
quasilinearity allows one to prove the main results \cref{thm:qlin_schon_main,thm:qlin_chow_main} by a (careful)
inductive procedure. Indeed, a key step in proving \cref{thm:qlin_schon_main,thm:qlin_chow_main} is to prove the
following main properties of quasilinear fans.

\begin{theorem}[{\cref{thm:prod_qlin,thm:star_qlin,thm:qlin_poincare}}]
    \begin{enumerate}
        \item Let $\Sigma$ and $\Sigma'$ be two reduced tropical fans. Then $\Sigma \times \Sigma'$ is quasilinear if
            and only if both $\Sigma$ and $\Sigma'$ are quasilinear.
        \item Star fans of quasilinear tropical fans are quasilinear.
        \item Quasilinear tropical fans are Poincar\'e. (In fact, quasilinear tropical fans satisfy the K\"ahler
            package.)
    \end{enumerate}
\end{theorem}

\subsection{Examples and applications}

The latter half of this article is devoted to giving criteria for checking whether a given algebraic variety is
quasilinear, and applying this criteria to describe some interesting nontrivial examples. Our main criteria are the
following.

\begin{theorem}[{\cref{thm:qlin_criterion_main}}]
    Let $Y \subset T$ be a quasilinear variety, and $f$ a regular function on $Y$ such that either $f$ is nonvanishing
    or $D = V(f) \subset Y$ is also quasilinear in $T$. Then $\wY = \{z=f\} \subset T \times \bC^*_z$ is quasilinear.
\end{theorem}

Recall that the moduli space $M_{0,n}$ of $n$ points in $\bP^1$ is a linear variety, the complement in $\bP^{n-3}$ of
the hyperplanes $x_i=0$ and $x_i=x_j$. In particular, the (previously known) analogues of
\cref{thm:qlin_schon_main,thm:qlin_chow_main} imply that the moduli space $\oM_{0,n}$ of stable $n$-pointed rational
curves is the log canonical compactification of $M_{0,n}$ (cf. \cite{keelGeometryChowQuotients2006}, \cite[Section
2]{hackingCompactificationModuliSpace2006}), as well as quickly recover Keel's results on the Chow ring of $\oM_{0,n}$
\cite{keelIntersectionTheoryModuli1992}.

There are two natural higher-dimensional analogues of $\oM_{0,n}$: the moduli space $\oM(r,n)$ of stable hyperplane
arrangements, compactifying the moduli space $M(r,n)$ of arrangements of $n$ hyperplanes in general position in
$\bP^{r-1}$ \cite{hackingCompactificationModuliSpace2006, alexeevModuliWeightedHyperplane2015}, and the moduli space
$\oY(E_n)$ of (weighted) stable marked del Pezzo surfaces of degree $9-n$ \cite{hackingStablePairTropical2009},
compactifying the moduli space $Y(E_n)$ of smooth marked del Pezzo surfaces of degree $9-n$. Using the notion of
quasilinearity, we obtain the analogues of the aforementioned results on $\oM_{0,n}$ for the first cases of these
higher-dimensional moduli spaces.

\begin{theorem}[{\cref{thm:M36_quasilinear,cor:M36_schon,cor:M36_chow}}]
    \begin{enumerate}
        \item The moduli space $M(3,6)$ of 6 lines in $\bP^2$ is quasilinear, and in particular sch\"on.
        \item The stable pair compactification $\oM(3,6)$ of $M(3,6)$ is the log canonical compactification.
        \item Let $\oM^{\Sigma}(3,6) \subset X(\Sigma)$ be any tropical compactification of $M(3,6)$. Then there is an
            isomorphism of Chow rings $A^*(\oM^{\Sigma}(3,6)) \cong A^*(X(\Sigma))$. If $X(\Sigma)$ is nonsingular, then
            $\oM^{\Sigma}(3,6)$ is a resolution of singularities of $\oM(3,6)$, and $A^*(\oM^{\Sigma}(3,6)) \cong
            H^*(\oM^{\Sigma}(3,6))$.
    \end{enumerate}
\end{theorem}

\begin{remark}
    Keel and Tevelev have conjectured that (the normalization of the main irreducible component of) $\oM(r,n)$ is the
    log canonical compactification of $M(r,n)$ precisely in the cases $(r,n)=(2,n)$ (in which case
    $\oM(2,n)=\oM_{0,n}$), and $(r,n)=(3,6),(3,7),(3,8)$ (as well as those cases obtained by duality $\oM(r,n) \cong
    \oM(n-r,n)$). This conjecture has recently been proven by work of Luxton
    \cite{luxtonLogCanonicalCompactification2008}, Corey \cite{coreyInitialDegenerationsGrassmannians2021}, and
    Corey-Luber \cite{coreyGrassmannianPlanesMathbb2023}. The above theorem gives an alternative proof of the main
    step in Luxton's argument for $\oM(3,6)$. (For another alternative argument in this case, see \cite[Section
    7]{schockIntersectionTheoryStable2022}.)

    The work of Corey and Luber \cite{coreyGrassmannianPlanesMathbb2023} implies that the moduli space $M(3,8)$ is not
    quasilinear, in contrast to an erroneous statement in the first preprint version of this article. On the other hand,
    the moduli space $M(3,7)$ can still be shown to be quasilinear, although the argument is quite involved---see
    \cite[Chapter 7]{schockGeometryTropicalCompactifications2022} for details.
\end{remark}

\begin{remark}
    Part (3) of the above theorem also quickly recovers the main results of \cite{schockIntersectionTheoryStable2022}.
\end{remark}

\begin{theorem}[{\cref{thm:Y_quasilinear,cor:Y_schon,cor:Y_chow}}]
    \begin{enumerate}
        \item The moduli space $Y(E_6)$ of smooth marked cubic surfaces is quasilinear, and in particular sch\"on.
        \item Let $\oY^{\Sigma}(E_6) \subset X(\Sigma)$ be any tropical compactification of $Y(E_6)$. Then
            $A^*(\oY^{\Sigma}(E_6)) \cong A^*(X(\Sigma))$. If $X(\Sigma)$ is nonsingular, then so is
            $\oY^{\Sigma}(E_6)$, and $A^*(\oY^{\Sigma}(E_6)) \cong H^*(\oY^{\Sigma}(E_6))$.
    \end{enumerate}
\end{theorem}

\begin{remark}
    The log canonical compactification $\oY(E_6)$ of $Y(E_6)$ is described in \cite{hackingStablePairTropical2009},
    following work of Naruki \cite{narukiCrossRatioVariety1982}.  It is not the the moduli space of stable marked cubic
    surfaces, but it is a tropical compactification, and admits an interpretation as a moduli space of \emph{weighted}
    stable marked cubic surfaces \cite{gallardoGeometricInterpretationToroidal2021, schockModuliWeightedStable2024}. The
    moduli space of stable marked cubic surfaces is also an explicitly described tropical compactification of $Y(E_6)$
    \cite{hackingStablePairTropical2009}. In particular, the above theorem allows one to write down explicit
    presentations of the Chow rings of moduli of weighted stable marked cubic surfaces. The intersection theory of
    $\oY(E_6)$ has also previously been studied in \cite{colomboChowGroupModuli2004}.
\end{remark}

\subsection{Outline}

This paper is organized as follows.

In \cref{sec:trop_prelims} we study the geometry of tropical fans. This section is largely a self-contained review of
known results and definitions on tropical fans, following mainly \cite{gathmannTropicalFansModuli2009,
allermannFirstStepsTropical2010, aminiHomologyTropicalFans2021}. We extend previous results by studying tropical
intersection theory (and in particular Poincar\'e duality) for possibly non-simplicial fans. Our hope is that this
section, while lengthy, could serve as a coherent introduction to tropical fans and their intersection theory for
readers with little to no prior background in this area.

In \cref{sec:qlin_fans,sec:qlin_vars,sec:qlin_comps}, we define respectively quasilinear fans, varieties, and
compactifications and prove the main results concerning them. \cref{sec:qlin_comps} also contains a quite general
discussion on the Chow rings of tropical compactifications, including some criteria for isomorphisms of Chow rings which
potentially hold outside of the quasilinear case, and may be of independent interest.

In \cref{sec:qlin_criteria} we study criteria for determining when a given closed subvariety of a torus is quasilinear,
and apply this criteria to give some basic examples and non-examples of quasilinear varieties.

Finally, in \cref{sec:applications} we apply the notion of quasilinearity to study the moduli spaces $M(3,6)$ of 6 lines
in $\bP^2$ and $Y(E_6)$ of marked cubic surfaces, proving the last of the main results.

\begin{acknowledgements}
    This work has benefited greatly from conversations with Valery Alexeev, Omid Amini, Daniel Corey, Emanuele Delucchi,
    Alex Fink, and Diane Maclagan. I am also grateful to Mathias Schulze for helpful questions and feedback, and to the
    anonymous referees for very helpful feedback, and for suggesting the terminology ``linearly stratified'' for
    varieties which are linear in the sense of \cite{totaroChowGroupsChow2014}---improving the previous name of ``weakly
    linear.''
\end{acknowledgements}

\section{Tropical preliminaries} \label{sec:trop_prelims}

\subsubsection{Notation}

Throughout we fix a lattice $N$ with dual $M = \Hom(N,\bZ)$, and write $N_{\bR} = N \otimes \bR \cong \bR^n$ (likewise
for $M_{\bR}$). We will often denote $N_{\bR}$ just by $\bR^n$. By a fan in $N_{\bR}$ we mean a strongly convex,
rational, polyhedral fan, so that all fans correspond to toric varieties \cite[Section
1.4]{fultonIntroductionToricVarieties1993}. We write $\Sigma_k$ for the set of $k$-dimensional cones of $\Sigma$.  We
write $\tau \prec \sigma$ if a cone $\tau$ is a face of a cone $\sigma$. We write $N_{\sigma}$ for the sublattice of $N$
generated by the cone $\sigma$. If $\tau \prec \sigma$ and $\dim \sigma = \dim \tau + 1$, then we write
$n_{\sigma,\tau}$ for any lattice point in the relative interior of $\sigma$ whose image generates the 1-dimensional
quotient lattice $N_{\sigma}/N_{\tau}$. The star fan of $\Sigma$ at a cone $\sigma$ is the fan in
$N_{\bR}/N_{\sigma,\bR}$ whose cones are the images of the cones of $\Sigma$ containing $\sigma$.

We denote by $X(\Sigma)$ the toric variety corresponding to a fan $\Sigma$, with torus $T = N \otimes \bC^*$. We write
$O(\sigma)$ for the torus orbit of $X(\Sigma)$ corresponding to the cone $\sigma$, and $V(\sigma)$ for its closure;
recall that $V(\sigma)$ is isomorphic to the toric variety $X(\Sigma^{\sigma})$.

\subsection{Basic definitions}

\begin{definition}[{\cite{gathmannTropicalFansModuli2009,
  allermannFirstStepsTropical2010}}]
  A \emph{tropical fan} in $N_{\bR}$ is a pair $(\Sigma,\omega)$ consisting of a pure $d$-dimensional fan $\Sigma$ in
  $N_{\bR}$, together with a function $\omega : \Sigma_d \to \bZ_{> 0}$ satisfying the \emph{balancing condition}: for
  every $(d-1)$-dimensional cone $\tau$, one has
  \[
    \sum_{\substack{\sigma \in \Sigma_d \\ \sigma \succ \tau}} \omega(\sigma)n_{\sigma,\tau} = 0 \mod{N_{\tau}}.
  \]
  The function $\omega$ is called the \emph{weight} of the tropical fan.
\end{definition}

\begin{example}
    Any complete fan is tropical, with weight one on all top-dimensional cones.
\end{example}

\begin{example}
    If $(\Sigma,\omega)$ and $(\Sigma',\omega')$ are two tropical fans, then $(\Sigma \times \Sigma',\omega \times
    \omega')$ is also tropical, where $(\omega \times \omega')(\sigma \times \sigma') = \omega(\sigma)\omega(\sigma')$
    \cite[Example 2.9]{gathmannTropicalFansModuli2009}.
\end{example}

Any refinement $\wt\Sigma$ of a tropical fan $(\Sigma,\omega)$ is also tropical, with the natural weight function
$\wt\omega(\wt\sigma) = \omega(\sigma)$ for all $\wt\sigma \in \wt\Sigma_d$, where $\sigma \in \Sigma_d$ is the unique
inclusion-minimal cone of $\Sigma$ containing $\wt\sigma$. The fan $(\wt\Sigma,\wt\omega)$ is then called a
\emph{tropical refinement} of $(\Sigma,\omega)$. Two tropical fans are said to be \emph{tropically equivalent} if they
have a common tropical refinement; this is indeed an equivalence relation \cite[Example
2.11]{gathmannTropicalFansModuli2009}.

\begin{definition}[{\cite{allermannFirstStepsTropical2010}}]
  A \emph{tropical fan cycle} $(\cF,\omega)$ in $N_{\bR}$ is an equivalence class of tropical fans in $N_{\bR}$ up to
  common tropical refinement. A \emph{(tropical) fan supported on} $(\cF,\omega)$, or a \emph{(tropical) fan structure}
  on $(\cF,\omega)$ is any representative of $(\cF,\omega)$.
\end{definition}

We typically think of a tropical fan cycle $(\cF,\omega)$ as the support $\cF$ of a pure-dimensional fan, such that any
fan supported on $\cF$ is a tropical fan, with the weight $\omega$. Thus, for instance, when we refer to the support of
a tropical fan, we view it as a tropical fan cycle.

\begin{example} \label{ex:local_defs}
  If $(\Sigma,\omega)$ is a tropical fan of dimension $d$, and $\sigma \in \Sigma_k$, then the star fan
  $\Sigma^{\sigma}$ is a tropical fan of dimension $d-k$, with the natural weight $\omega^{\sigma}$ inherited from
  $\omega$.

  If $(\cF,\omega)$ is a tropical fan cycle, and $v \in \cF$ (viewing $\cF$ as the support of a fan in $N_{\bR}$, as
  above), then set \cite[A.6]{gublerGuideTropicalizations2013}
  \[
    \cF^v = \{u \in N_{\bR} \mid u + \epsilon v \in \cF \text{ for all
    sufficiently small } \epsilon\}.
  \]
  Then $\cF^v$ is also a tropical fan cycle with weight inherited from $\omega$. If $\Sigma$ is a fan structure on $\cF$
  such that $v$ is in the relative interior of a cone $\sigma \in \Sigma_k$, then
  \[
      \cF^v = \lvert \Sigma^{\sigma} \rvert \times \bR^k.
  \]
\end{example}

\begin{definition}
  A tropical fan $(\Sigma,\omega)$ is \emph{reduced} if $\omega(\sigma)=1$ for all top-dimensional cones $\sigma$ of
  $\Sigma$.
\end{definition}

\begin{definition}
  A tropical fan $(\Sigma,\omega)$ is \emph{irreducible} if there is no tropical fan $(\Sigma',\omega')$ of the same
  dimension, such that $\lvert \Sigma' \rvert \subsetneq \lvert \Sigma \rvert$. A tropical fan $(\Sigma,\omega)$ is
  \emph{locally irreducible} if all of its star fans are irreducible.
\end{definition}

\begin{definition} \label{def:nice_props}
    \begin{itemize}
        \item A property $\cP$ of tropical fans is \emph{intrinsic to the support} if whenever $(\Sigma,\omega)$ and
            $(\Sigma',\omega')$ are two representatives of the same tropical fan cycle, $(\Sigma,\omega)$ satisfies
            $\cP$ if and only if $(\Sigma',\omega')$ satisfies $\cP$.
        \item A property $\cP$ of tropical fans is \emph{local} if whenever $(\Sigma,\omega)$ satisfies $\cP$, all of
            its star fans (viewed as tropical fans by \cref{ex:local_defs}) also satisfy $\cP$.
        \item A property $\cP$ of tropical fans which is intrinsic to the support is \emph{stably invariant} if a
            tropical fan $(\Sigma,\omega)$ satisfies $\cP$ if and only if $\Sigma \times \Delta$ satisfies $\cP$ for all
            complete fans $\Delta$ of dimension $\geq 1$ \cite[3.2.3]{aminiHomologyTropicalFans2021}.
    \end{itemize}
\end{definition}

The point of stable invariance is to ensure properties of tropical fans which are intrinsic to the support and local
also give well-behaved local properties of tropical fan cycles. Indeed, continuing with the notation of
\cref{ex:local_defs}, if $v \in \cF$ and $\Sigma,\Sigma'$ are two fan structures on $\cF$ such that $v$ is in the
relative interior of $\sigma \in \Sigma'$ and $\sigma' \in \Sigma'$, then
\[
    \cF^v = \lvert \Sigma^{\sigma} \rvert \times \bR^{\dim \sigma} = \lvert (\Sigma')^{\sigma'} \rvert \times \bR^{\dim
    \sigma'},
\]
so in order to, for instance, obtain a property $\cP$ for $(\Sigma')^{\sigma'}$ from the corresponding property for
$\Sigma^{\sigma}$, we ask that the property is also satisfied by $\lvert \Sigma^{\sigma} \rvert \times \bR^{\dim \sigma
- \dim \sigma'}$.

Following the above discussion, suppose $\cP$ is a property of tropical fans which is intrinsic to the support, local,
and stably invariant. Then a tropical fan $\Sigma$ supported on a tropical fan cycle $\cF$ satisfies $\cP$ if and only
if \emph{any} tropical fan supported on $\cF$ satisfies $\cP$. Thus in this case we can also say that the tropical fan
cycle $\cF$ satisfies the property $\cP$.

\begin{proposition} \label{prop:red_loc_irred_nice}
    \begin{enumerate}
        \item The property of being reduced is intrinsic to the support, local, and stably invariant.
        \item The property of being irreducible is intrinsic to the support and stably invariant. In fact, if $\Sigma$
            and $\Sigma'$ are two tropical fans, then $\Sigma \times \Sigma'$ is irreducible if and only if $\Sigma$ and
            $\Sigma'$ are both irreducible.
        \item The property of being locally irreducible is intrinsic to the support, local, and stably invariant.
    \end{enumerate}
\end{proposition}

\begin{proof}
    Immediate.
\end{proof}

In particular, by the preceding discussion, it makes sense to speak of tropical fan \emph{cycles} being reduced or
(locally) irreducible.

\begin{remark}
    We warn the reader that while our definition of irreducibility given above agrees with the one given in
    \cite{gathmannTropicalFansModuli2009, gathmannIrreducibleCyclesPoints2012}, our definition of \emph{local}
    irreducibility differs from that of \cite[Definition 2.21]{gathmannIrreducibleCyclesPoints2012}, which we might
    instead call \emph{irreducible in codimension one}.
\end{remark}

\subsection{Intersection theory}

\begin{definition}
    Let $\Sigma$ be any fan. The \emph{Chow ring (or ring of tropical cocycles)} $A^*(\Sigma)$ of $\Sigma$ is the
    (operational \cite[Chapter 17]{fultonIntersectionTheory1998}) Chow ring $A^*(X(\Sigma))$ of the corresponding toric
    variety $X(\Sigma)$.
\end{definition}

We recall the following well-known presentation of the Chow ring of a unimodular fan.

\begin{theorem}[{\cite{brionPiecewisePolynomialFunctions1996, katzPiecewisePolynomialsMinkowski2008}}]
    \label{thm:chow_pres}
    Let $\Sigma$ be a unimodular fan in $N_{\bR}$. Then
    \[
        A^*(\Sigma) \cong \frac{\bZ[x_{\rho} \mid \rho \in \Sigma_1]}{L + I},
    \]
    where
    \begin{itemize}
        \item $L$ is the ideal of \emph{linear relations}, generated by $\sum_{\rho \in \Sigma_1} \langle u,v_{\rho}
            \rangle x_{\rho}$ for $u \in M$.
        \item $I$ is the \emph{Stanley-Reisner} ideal, generated by $x_{\rho_1} \cdots x_{\rho_k}$ whenever
            $\rho_1,\ldots,\rho_k$ do not span a cone of $\Sigma$.
    \end{itemize}
    If $\Sigma$ is simplicial, the same presentation holds with $\bQ$-coefficients.
\end{theorem}

\begin{definition}[{\cite{fultonIntersectionTheoryToric1997}}]
    Let $\Sigma$ be any fan. The \emph{group of $k$-dimensional Minkowski weights (or $k$-dimensional tropical fan
    cycles) on $\Sigma$} is the group $M_k(\Sigma)$ of functions $\omega : \Sigma_k \to \bZ$ satisfying the
    \emph{balancing condition}: for any cone $\tau \in \Sigma_{k-1}$,
    \[
        \sum_{\substack{\sigma \in \Sigma_k \\ \sigma \succ \tau}} \omega(\sigma)n_{\sigma,\tau} = 0 \mod{N_{\tau}}.
    \]
\end{definition}

In particular, in the definition of a tropical fan $(\Sigma,\omega)$, the weight $\omega$ is nothing more than a
$d$-dimensional Minkowski weight on $\Sigma$. The group of $k$-dimensional Minkowski weights on $\Sigma$ is thus
naturally interpreted as the group of $k$-dimensional tropical fan subcycles of $\Sigma$, \emph{with fan structures
induced by $\Sigma$}.

It is a standard fact that for any toric variety $X(\Sigma)$, the Chow group $A_{n-k}(X(\Sigma))$ is generated by the
classes of the torus orbit closures $V(\sigma)$ for $\sigma \in \Sigma_k$, see \cref{cor:chow_groups}. The relations on
$A_{n-k}(X(\Sigma))$ are described by the balancing condition, leading to the following fundamental description of
$M_k(\Sigma)$.

\begin{lemma}[{\cite[Proposition 2.1]{fultonIntersectionTheoryToric1997}}] \label{lem:mink_chow}
    For any fan $\Sigma$ in $N_{\bR} \cong \bR^n$, one has
    \[
        M_k(\Sigma) \cong \Hom(A_{n-k}(X(\Sigma)),\bZ).
    \]
\end{lemma}

\begin{definition}
    Let $(\Sigma,\omega)$ be a tropical fan. Define the \emph{(tropical) cap product}
    \begin{align*}
        A^k(\Sigma) \times M_j(\Sigma) &\to M_{j-k}(\Sigma), \\
        (\alpha,\eta) \mapsto \alpha \cap \eta
    \end{align*}
    as
    \begin{align*}
        A^k(X(\Sigma)) \times \Hom(A_{n-j}(X(\Sigma)),\bZ) \to \Hom(A_{n-j+k}(X(\Sigma)),\bZ), \\
        (\alpha,\eta) \mapsto (\beta \mapsto \eta(\alpha \cap \beta)),
    \end{align*}
    where $\alpha \cap \beta \in A_{n-j}(X(\Sigma))$ is the usual cap product between $\alpha \in A^k(X(\Sigma))$ and
    $\beta \in A_{n-j+k}(X(\Sigma))$.
\end{definition}

\begin{remark}
    When $\Sigma$ is unimodular or simplicial, $A^*(\Sigma)$ can be described in terms of piecewise polynomials, and
    there is a purely combinatorial interpretation of the above cap product, see for instance
    \cite{francoisCocyclesTropicalVarieties2012}. However, if $\Sigma$ is not simplicial, then piecewise polynomials do
    not give the correct notion for $A^*(\Sigma)$, cf. \cite{katzPiecewisePolynomialsMinkowski2008}.
\end{remark}

\begin{remark}
    For a tropical fan cycle $\cF$, one may define $A^*(\cF)$ (resp. $M_*(\cF)$) as the direct limit over all fan structures
    $\Sigma$ on $\cF$ of $A^*(\Sigma)$ (resp. $M_*(\Sigma)$), cf. e.g. \cite{grossCyclesCocyclesDuality2021}. This will
    not play a significant role for us.
\end{remark}

\begin{definition}
    A tropical fan $(\Sigma,\omega)$ of dimension $d$ is \emph{Poincar\'e} if $- \cap \omega : A^k(\Sigma) \to
    M_{d-k}(\Sigma)$ is an isomorphism for all $k$. The tropical fan $(\Sigma,\omega)$ is \emph{star-Poincar\'e} if all
    of its star fans are Poincar\'e.
\end{definition}

\begin{remark}
    If $\Sigma$ is unimodular, then the toric variety $X(\Sigma)$ is nonsingular, so $A_{n-d+k}(X(\Sigma)) \cong
    A^{d-k}(\Sigma)$. Poincar\'e duality in this case is therefore equivalent to the Chow ring $A^*(\Sigma)$ being a
    Poincar\'e duality ring of dimension $d$.
\end{remark}

\begin{example}
    Any complete fan is star-Poincar\'e by \cite{fultonIntersectionTheoryToric1997}.
\end{example}

Minkowski weights lead to a very useful characterization of irreducible tropical fans.

\begin{proposition}[{\cite[Lemma 2.20]{gathmannIrreducibleCyclesPoints2012}}] \label{prop:irred_criterion}
    A tropical fan $(\Sigma,\omega)$ is irreducible if and only if $M_d(\Sigma) \cong \bZ$.
\end{proposition}

\begin{definition}
    The \emph{fundamental weight} of an irreducible tropical fan $\Sigma$ is the (unique) positive generator of
    $M_d(\Sigma) \cong \bZ$.
\end{definition}

\begin{lemma}
    If $(\Sigma,\omega)$ is a Poincar\'e tropical fan, then $\Sigma$ is irreducible and $\omega$ is the fundamental
    weight. If $(\Sigma,\omega)$ is star-Poincar\'e, then it is reduced and locally irreducible.
\end{lemma}

\begin{proof}
    Since for any fan $\Sigma$, the corresponding toric variety $X(\Sigma)$ is irreducible, one has $A^0(\Sigma) \cong
    \bZ$. Poincar\'e duality implies that $A^0(\Sigma) \to M_d(\Sigma)$, $1 \mapsto \omega$ is an isomorphism, hence
    $\Sigma$ is irreducible by \cref{prop:irred_criterion} and $\omega$ is the fundamental weight. It follows
    immediately that a star-Poincar\'e tropical fan is locally irreducible. If $\sigma$ is a top-dimensional cone of
    $\Sigma$, then $\Sigma^{\sigma}$ is a $0$-dimensional fan, so $M_0(\Sigma^{\sigma}) \cong \bZ$, generated by the
    weight $w(0)=1$. In particular, if $\Sigma$ is star-Poincar\'e at $\sigma$, then $\omega(\sigma)=w(0)=1$.
\end{proof}

\begin{proposition} \label{prop:chow_prods}
    Let $(\Sigma,\omega),(\Sigma',\omega')$ be two tropical fans. Then $A^*(\Sigma \times \Sigma') \cong A^*(\Sigma)
    \otimes A^*(\Sigma')$ and $M_*(\Sigma \times \Sigma') \cong M_*(\Sigma) \otimes M_*(\Sigma')$. In particular,
    $\Sigma \times \Sigma'$ is Poincar\'e, if and only if $\Sigma$ and $\Sigma'$ are both Poincar\'e.
\end{proposition}

\begin{proof}
    The first part follows from the well-known fact that toric varieties satisfy the \emph{Chow-K\"unneth property}: if
    $X(\Sigma)$ is a toric variety and $Z$ is any finite-type scheme, then $A_*(X(\Sigma)) \otimes A_*(Z) \to
    A_*(X(\Sigma) \times Z)$ is an isomorphism \cite[Theorem 2]{fultonIntersectionTheorySpherical1995}. (See also
    \cref{lem:weakly_linear} below.)

    For the second part, say $\dim \Sigma = d$, $\dim \Sigma' = d'$.
    Then
    \[
        A^k(\Sigma \times \Sigma') \cong \bigoplus_{i+j=k} A^i(\Sigma) \otimes A^j(\Sigma')
    \]
    and
    \[
        M_{d+d'-k}(\Sigma \times \Sigma') \cong \bigoplus_{i+j=k} M_{d-i}(\Sigma) \otimes M_{d'-j}(\Sigma'),
    \]
    and the cap product with the weight $\omega \times \omega'$ is given by
    \[
        \bigoplus_{i+j=k} A^i(\Sigma) \otimes A^j(\Sigma') \xrightarrow{\bigoplus (- \cap \omega, - \cap \omega')}
        \bigoplus_{i+j=k} M_{d-i}(\Sigma) \otimes M_{d'-j}(\Sigma').
    \]
    It follows that $\Sigma \times \Sigma'$ is Poincar\'e $\iff$ $\Sigma$ and $\Sigma'$ are both Poincar\'e.
\end{proof}

\begin{theorem} \label{thm:poincare_intrinsic}
    The property of being a star-Poincar\'e tropical fan is intrinsic to the support.
\end{theorem}

\begin{proof}
    This is by now a well-known (and nontrivial) result in the unimodular case, see
    \cite{adiprasitoHodgeTheoryCombinatorial2018, ardilaLagrangianGeometryMatroids2022, aminiHomologyTropicalFans2021,
    grossCyclesCocyclesDuality2021}. The proof uses the weak factorization theorem for toric varieties
    \cite{wlodarczykDecompositionBirationalToric1997}. With our definition of Poincar\'e duality for non-unimodular
    fans, a similar proof also works for the non-unimodular case, cf. \cite[Proposition
    3.5]{grossCyclesCocyclesDuality2021}.
\end{proof}

\begin{corollary} \label{cor:poincare_nice}
    The property of being a star-Poincar\'e tropical fan is intrinsic to the support, local, and stably invariant.
\end{corollary}

\begin{proof}
    ``Local'' is by definition, while ``stably invariant'' and ``intrinsic to the support'' follow by
    \cref{prop:chow_prods,thm:poincare_intrinsic}.
\end{proof}

In particular, it makes sense to speak of star-Poincar\'e tropical fan cycles.

\subsection{Morphisms} \label{sec:morphisms}

For a fan $\Sigma \subset N_{\bR}$ (resp. fan cycle $\cF$), let $N_{\Sigma,\bR}$ (resp. $N_{\cF,\bR}$) be the vector
subspace of $N_{\bR}$ spanned by $\Sigma$ (resp $\cF$), and let $N_{\Sigma} = N \cap N_{\Sigma,\bR}$ (resp. $N_{\cF} = N
\cap N_{\cF,\bR}$).

By definition, $N_{\Sigma}$ is spanned by a subset of a basis of $N$, so we do not need to worry about finite index
sublattices (cf. \cite[Proof of Proposition 3.3.9]{coxToricVarieties2011}). We can view $\Sigma$ either as a fan in
$N_{\Sigma}$ or as a fan in $N$; write $X(\Sigma;N_{\Sigma})$ and $X(\Sigma;N)$ for the corresponding toric varieties.
Extending a basis of $N_{\Sigma}$ to a basis of $N$ implies that \cite[Proposition 3.3.9]{coxToricVarieties2011}
\[
    X(\Sigma;N) \cong X(\Sigma;N_{\Sigma}) \times (\bC^*)^{\rk N - \rk N_{\Sigma}}.
\]
This implies in particular that $A^*(X(\Sigma;N)) \cong A^*(X(\Sigma;N_{\Sigma}))$ and $A_*(X(\Sigma;N)) \cong
A_*(X(\Sigma;N_{\Sigma}))$, so both $A^*(\Sigma)$ and $M_*(\Sigma)$ are independent of the choice of ambient space
between $N_{\Sigma}$ and $N$.

\begin{definition}
    Let $\Sigma \subset N_{\bR}$ and $\Sigma' \subset N'_{\bR}$ be two fans. A \emph{morphism} $f : \Sigma \to \Sigma'$
    is a map $f : \lvert \Sigma \rvert \to \lvert \Sigma' \rvert$ induced by a linear map $f : N_{\Sigma} \to
    N_{\Sigma'}$, such that $f(\sigma)$ is contained in a cone of $\Sigma'$ for all $\sigma \in \Sigma$ \cite[Section
    3.3]{coxToricVarieties2011}.

    Let $\cF \subset N_{\bR}$ and $\cF' \subset N'_{\bR}$ be two tropical fan cycles. A \emph{morphism} $f : \cF \to
    \cF'$ is a map $f : \cF \to \cF'$ induced by a morphism of fans $\Sigma \to \Sigma'$ for some fan structures
    $\Sigma$ on $\cF$ and $\Sigma'$ on $\cF'$.
\end{definition}

\begin{remark}
    There is no condition on the fundamental weights for a morphism of tropical fans or tropical fan cycles, cf.
    \cite[Definition 4.1]{allermannFirstStepsTropical2010}.
\end{remark}

\begin{definition}
    Let $f : (\Sigma,\omega) \to (\Sigma',\omega')$ be a morphism of tropical fans. The \emph{pullback of tropical fan
    cocycles} is the usual pullback morphism $f^* : A^*(\Sigma') = A^*(X(\Sigma')) \to A^*(X(\Sigma)) = A^*(\Sigma)$.
\end{definition}

\begin{definition}
    Let $f : (\Sigma,\omega) \to (\Sigma',\omega')$ be a morphism of tropical fans, and assume that the image of each
    cone of $\Sigma$ is a cone of $\Sigma'$. The \emph{pushforward of tropical fan cycles} is the morphism $f_* :
    M_k(\Sigma) \to M_k(\Sigma')$ sending $\eta \in M_k(\Sigma)$ to $f_*\eta \in M_k(\Sigma')$, defined by
    \[
        f_*\eta(\sigma') = \sum_{\substack{\sigma \in \Sigma_k \\ f(\sigma)=\sigma'}} \eta(\sigma) [N_{\sigma'} :
        f(N_{\sigma})].
    \]
\end{definition}

\begin{remark}
    Recall the Chow groups of an arbitrary toric variety are generated by the classes of the torus orbit closures.  If
    $f : (\Sigma,\omega) \to (\Sigma',\omega')$ is a morphism of tropical fans such that the image of each cone of
    $\Sigma$ is a cone of $\Sigma'$, then all fibers of the corresponding morphism of toric varieties $f : X(\Sigma) \to
    X(\Sigma')$ have the same dimension \cite[Lemma 4.1]{abramovichWeakSemistableReduction2000}. Therefore one can
    define a pullback morphism $f^* : A_{n'-k}(X(\Sigma')) \to A_{n-k}(X(\Sigma))$ by
    \[
        f^*[V(\sigma')] = [f^{-1}(V(\sigma'))] = \sum_{\substack{\sigma \in \Sigma_k \\ f(\sigma) = \sigma'}}
        [V(\sigma)] \cdot [N_{\sigma'} : f(N_{\sigma})],
    \]
    and $f_* : M_k(\Sigma) \to M_k(\Sigma')$ is the dual to this pullback under the isomorphism of \cref{lem:mink_chow},
    cf. \cite[Proposition 3.7, Corollaries 4.6, 4.7]{fultonIntersectionTheoryToric1997}.

    If $\Sigma'$ is unimodular then $f : X(\Sigma) \to X(\Sigma')$ is flat \cite[Exercise
    III.10.9]{hartshorneAlgebraicGeometry1977}, and the above pullback is the usual flat pullback as in \cite[Section
    1.7]{fultonIntersectionTheory1998}.
\end{remark}

\begin{lemma}[Projection formula]
    Let $f : (\Sigma,\omega) \to (\Sigma',\omega')$ be a morphism of tropical fans such that the image of every cone of
    $\Sigma$ is a cone of $\Sigma'$. Let $\alpha \in A^k(\Sigma')$ and $\eta \in M_j(\Sigma)$. Then
    \[
        f_*(f^*\alpha \cap \eta) = \alpha \cap f_*\eta.
    \]
\end{lemma}

\begin{proof}
    Let $\beta \in A_{n-j+k}(X(\Sigma'))$. Then
    \begin{align*}
        (\alpha \cap f_*\eta)(\beta) &= f_*\eta(\alpha \cap \beta) \\
                                     &= \eta(f^*(\alpha \cap \beta)) \\
                                     &= \eta(f^*\alpha \cap f^*\beta) \\
                                     &= (f^*\alpha \cap \eta)(f^*\beta) \\
                                     &= f_*(f^*\alpha \cap \eta).
    \end{align*}
    (Here the second and last equality follow from the discussion of the above remark.)
\end{proof}

\begin{definition} \label{def:trop_iso}
    Let $(\Sigma,\omega) \subset N_{\bR}$ and $(\Sigma',\omega') \subset N'_{\bR}$ be two tropical fans. An
    \emph{isomorphism} $f : (\Sigma,\omega) \xrightarrow{\sim} (\Sigma',\omega')$ is a morphism induced by an
    isomorphism $f : N_{\Sigma} \xrightarrow{\sim} N_{\Sigma'}$, such that $f(\sigma)$ is a cone of $\Sigma'$ for all
    $\sigma \in \Sigma$, $f^{-1}(\sigma')$ is a cone of $\Sigma$ for all $\sigma' \in \Sigma'$, $f_*\omega = \omega'$,
    and $(f^{-1})_*\omega'=\omega$.

    Let $(\cF,\omega) \subset N_{\bR}$ and $(\cF',\omega') \subset N'_{\bR}$ be two tropical fan cycles. An
    \emph{isomorphism} $f : (\cF,\omega) \xrightarrow{\sim} (\cF',\omega')$ is a morphism $f : \cF \to \cF'$ induced by
    an isomorphism of tropical fans for some fan structures $\Sigma$ on $\cF$ and $\Sigma'$ on $\cF'$.
\end{definition}

\begin{remark} \label{rmk:iso_fans_to_toric}
    In our definition, an isomorphism of fans $\Sigma \subset N_{\bR}$ and $\Sigma' \subset N'_{\bR}$ does not
    necessarily induce an isomorphism of toric varietes $X(\Sigma;N) \cong X(\Sigma;N')$, but instead an isomorphism of
    toric varieties $X(\Sigma;N_{\Sigma}) \cong X(\Sigma';N_{\Sigma'})$. Note this implies in particular that isomorphic
    tropical fans have isomorphic rings of tropical cocycles and groups of tropical cycles.
\end{remark}

\subsection{Divisors and tropical modifications}

Tropical modifications, introduced by Mikhalkin \cite{mikhalkinTropicalGeometryIts2007}, play a fundamental role in this
article. They are used in, e.g., \cite{shawTropicalIntersectionProduct2013} to describe intersection theory of linear
tropical fans, and in \cite{aminiHomologyTropicalFans2021} to understand homology of tropical fans. Our exposition is
based on \cite{allermannFirstStepsTropical2010, aminiHomologyTropicalFans2021}.

\subsubsection{Divisors and piecewise integral linear functions}

\begin{definition}
    Let $\Sigma$ be a fan. A \emph{linear function} on $\Sigma$ is a continuous function $\ell : \lvert \Sigma \rvert
    \to \bR$ which is the restriction of an integral linear function $\ell \in M = \Hom(N,\bZ)$. A \emph{piecewise
    integral linear function} on $\Sigma$ is a continuous function $\varphi : \lvert \Sigma \rvert \to \bR$, such that
    on each cone $\sigma \in \Sigma$, $\varphi\lvert_{\sigma}$ is identified with the restriction of an integral linear
    function $\varphi_{\sigma} \in M_{\sigma} = \Hom(N_{\sigma},\bZ)$. The group of piecewise integral linear functions
    on $\Sigma$ is denoted by $PP^1(\Sigma)$.
\end{definition}

\begin{proposition}[{\cite[Theorem 4.5, Corollary 4.6]{katzPiecewisePolynomialsMinkowski2008}}] \label{prop:pic}
    Let $\Sigma$ be any fan. Then
    \[
        \Pic(X(\Sigma)) \cong A^1(\Sigma) \cong PP^1(\Sigma)/M.
    \]
\end{proposition}

\begin{definition}[{\cite{allermannFirstStepsTropical2010}}]
    Let $\Sigma$ be a tropical fan. A \emph{tropical Cartier divisor} on $\Sigma$ is an element of $A^1(\Sigma)$.
\end{definition}

\begin{remark}
    While on the algebraic side $A^1(X(\Sigma))$ is really the group of Cartier divisor \emph{classes}, there is no such
    distinction on the tropical side. This is because, linear functions on $\Sigma$ actually define \emph{trivial}
    tropical divisors, as opposed to potentially nontrivial principal Cartier divisors on $X(\Sigma)$. See
    \cite{allermannFirstStepsTropical2010, allermannRationalEquivalenceTropical2016a} for more details.
\end{remark}

\begin{definition}
    Let $\Sigma$ be a tropical fan of dimension $d$. A \emph{tropical Weil divisor} on $\Sigma$ is an element of
    $M_{d-1}(\Sigma)$.
\end{definition}

\begin{definition}[{\cite{allermannFirstStepsTropical2010}}]
  \label{def:ord_vanishing}
    Let $\varphi$ be a piecewise integral linear function on a $d$-dimensional tropical fan $(\Sigma,\omega)$. The
    \emph{order of vanishing} of $\varphi$ along a cone $\tau \in \Sigma_{d-1}$ is
    \begin{align*}
      \ord_{\tau}(\varphi) = \varphi_{\tau}\left( \sum_{\substack{\sigma \in \Sigma_d \\ \sigma \succ \tau}}
      \omega(\sigma)n_{\sigma,\tau}\right) - \sum_{\substack{\sigma \in \Sigma_d \\ \sigma \succ \tau}}
      \varphi_{\sigma}(\omega(\sigma)n_{\sigma,\tau})
    \end{align*}
\end{definition}

\begin{lemma}[{\cite[Proposition 3.7]{allermannFirstStepsTropical2010}}]
    The function $\Sigma_{d-1} \to \bZ$, $\tau \mapsto \ord_{\tau}(\varphi)$ is a well-defined (independent of the
    choice of $n_{\sigma,\tau}$) $(d-1)$-dimensional Minkowski weight on $\Sigma$.
\end{lemma}

\begin{definition}[{\cite[Definition 3.4]{allermannFirstStepsTropical2010}}]
  \label{def:principal_weil}
  Let $\varphi$ be a piecewise integral linear function on a $d$-dimensional tropical fan $(\Sigma,\omega)$. The
  \emph{principal tropical Weil divisor} associated to $\varphi$ is the tropical fan cycle $\div(\varphi) \in
  M_{d-1}(\Sigma)$ defined by $\div(\varphi)(\tau) = \ord_{\tau}(\varphi)$.
\end{definition}

When we wish to view $\div(\varphi)$ as a tropical fan rather than a tropical fan cycle, we write it as $\div(\varphi) =
(\Delta,\delta)$, where $\Delta$ is the support of $\div(\varphi)$, i.e the $(d-1)$-dimensional fan consisting of the
cones $\tau \in \Sigma_{d-1}$ for which $\ord_{\tau}(\varphi) \neq 0$, and $\delta(\tau) = \ord_{\tau}(\varphi)$.  Thus
by abuse of notation $\div(\varphi)$ could refer to either the tropical fan cycle $\div(\varphi) \in M_{d-1}(\Sigma)$ or
the tropical fan $(\Delta,\delta)$.

\begin{proposition}[{\cite[Proposition 4.8]{aminiHomologyTropicalFans2021}}]
    Let $(\Sigma,\omega)$ be a $d$-dimensional tropical fan and let $\varphi \in A^1(\Sigma)$. Then
    \[
      \varphi \cap \omega = -\div(\varphi)
    \]
\end{proposition}

\begin{proof}
    This is proved in the reduced unimodular case in \cite[Proposition 4.8]{aminiHomologyTropicalFans2021}. A similar
    proof works in the general case, using in particular that the description of elements of $A^1(\Sigma)$ as piecewise
    integral linear functions modulo globally linear functions holds for arbitrary fans $\Sigma$ (\cref{prop:pic}).
\end{proof}

\begin{remark}
    Our definition of $\ord_{\tau}(\varphi)$ agrees with the one in \cite{aminiHomologyTropicalFans2021} (in the reduced
    case), but is the negative of the one in \cite{allermannFirstStepsTropical2010}. This is because we work with the
    ``min'' convention in tropical geometry, while \cite{allermannFirstStepsTropical2010} works with the ``max''
    convention.
\end{remark}

\subsubsection{Tropical modifications}

The definition of the tropical divisor associated to a piecewise integral linear function is motivated by the following
observation.

Let $\varphi : \Sigma \to \bR$ be a piecewise integral linear function on a tropical fan $(\Sigma,\omega)$ in $N_{\bR}$. The
graph of $\varphi$ gives a fan $\Gamma_\varphi(\Sigma) \subset \wN_{\bR} = N_{\bR} \times \bR$ with cones
\[
  \wt\sigma = \{(x,\varphi(x)) \mid x \in \sigma\}.
\]
A direct computation \cite[Construction 3.3]{allermannFirstStepsTropical2010} shows that, for a codimension one cone
$\tau$ of $\sigma$,
\[
    \sum_{\substack{\sigma \succ \tau, \\ \dim \sigma = \dim \tau + 1}} (n_{\sigma,\tau},\varphi_{\sigma}(n_{\sigma,\tau})) =
    \left(0,-\ord_{\tau}(\varphi)\right) \mod{N_{\tau}}.
\]
In particular, $\Gamma_\varphi(\Sigma)$ fails the balancing condition around the cones $\wt\tau$ for which
$\ord_{\tau}(\varphi)$ is nonzero. Balancing is restored by adding the cones
\[
    \tau_{\geq} = \wt\tau + \bR_{\geq 0}(\vec{0},1)
\]
for $\tau \in \div(\varphi)$, with weight $\ord_{\tau}(\varphi)$.

\begin{definition}
    The \emph{tropical modification} of a tropical fan $\Sigma \subset N_{\bR}$ with respect to a piecewise integral
    linear function $\varphi$ is the tropical fan $\cT\cM_{\varphi}(\Sigma) \subset \wN_{\bR} = N_{\bR} \times \bR$ with cones
    \begin{itemize}
      \item $\wt\sigma = \{(x,\varphi(x)) \mid x \in \sigma\}$, for $\sigma \in \Sigma$,
      \item $\tau_{\geq} = \wt\tau + \bR_{\geq 0}(\vec{0},1)$ for $\tau \in \div(\varphi)$,
    \end{itemize}
    and weights $\wt\omega(\wt\sigma) = \omega(\sigma)$ and $\wt\omega(\tau_{\geq}) = \ord_{\tau}(\varphi)$.
\end{definition}

Tropical modifications are indeed tropical fans by \cite[Proposition 3.7]{allermannFirstStepsTropical2010}---see
\cite[Section 4.5]{mikhalkinTropicalGeometry2018} and \cite[Sections 4-5]{aminiHomologyTropicalFans2021} (in the reduced
case) for more details. See \cref{sec:example_fans} for examples.

\begin{remark} \label{rmk:degen_mods}
    An important special case of tropical modifications occurs when the divisor $\Delta = \div(\varphi)$ is trivial, in
    which case the tropical modification is said to be \emph{degenerate} \cite[5.1.4]{aminiHomologyTropicalFans2021}. In
    general a degenerate tropical modification may differ from the original fan \cite[Example
    11.4]{aminiHomologyTropicalFans2021}; however, this does not occur if $\varphi$ is a \emph{globally} integral linear
    function, in which case the modification is simply a linear re-embedding of $\Sigma$, and therefore isomorphic to
    $\Sigma$. In particular, degenerate tropical modifications of Poincar\'e tropical fans are isomorphic to the
    original fan, see \cite[5.1.4]{aminiHomologyTropicalFans2021}.
\end{remark}

\begin{remark}
    Tropical modifications depend both on the divisor $\Delta$ and the piecewise integral linear function $\varphi$.  In
    \cite[Example 4.2]{brugalleInflectionPointsReal2012}, the authors gives an example of two different tropical
    modifications along the same divisor. On the other hand, if $\Sigma$ is Poincar\'e, then $\div : A^1(\Sigma) \to
    M_{d-1}(\Sigma)$ is an isomorphism, so any two piecewise integral linear functions which define the same tropical
    Weil divisor differ by a linear function. This linear function induces an isomorphism between the two tropical
    modifications, cf. \cite[5.1.5]{aminiHomologyTropicalFans2021}. Thus when $\Sigma$ is Poincar\'e, we write
    $\cT\cM_{\Delta}(\Sigma)$ to denote any tropical modification of $\Sigma$ with respect to a piecewise integral
    linear function $\varphi$ such that $\div(\varphi)=\Delta$.
\end{remark}

\begin{lemma}
    If $\Sigma$ is a unimodular tropical fan, then any tropical modification of $\Sigma$ is also unimodular.
\end{lemma}

\begin{proof}
    Clearly if $\sigma \in \Sigma$ is unimodular, then the cone $\wt\sigma$ is also unimodular. The possible new cones
    of $\Sigma$ have the form $\wt\tau + \bR_{\geq 0}(\vec{0},1)$ for some $\tau \in \Sigma$, hence these are clearly
    unimodular as well.
\end{proof}

\begin{lemma} \label{lem:red_mod}
    A tropical modification of a reduced tropical fan along a trivial or reduced tropical divisor is reduced.
\end{lemma}

\begin{proof}
    Immediate.
\end{proof}

\begin{proposition}[{\cite[Proposition 5.2]{aminiHomologyTropicalFans2021}}] \label{prop:star_fan_mod}
    Let $\varphi$ be a rational function on a tropical fan $\Sigma$, and let $\wt\Sigma = \cT\cM_{\varphi}(\Sigma)$. Let
    $\Delta = \div(\varphi)$ (possibly trivial). The star fans of $\wt\Sigma$ are described as follows.
    \begin{enumerate}
      \item If $\sigma \in \Sigma$, then $\wt\Sigma^{\wt\sigma} \cong \cT\cM_{\varphi^{\sigma}}(\Sigma^{\sigma})$, where
          $\varphi^{\sigma}$ is any piecewise integral linear function on $\Sigma^{\sigma}$ induced by $\varphi$. In
          particular, if $\sigma \not\in \Delta$, then $\wt\Sigma^{\wt\sigma}$ is a degenerate tropical modification of
          $\Sigma^{\sigma}$.
      \item If $\tau \in \Delta$, then $\wt\Sigma^{\tau_{\geq}} \cong \Delta^{\tau}$.
    \end{enumerate}
\end{proposition}

\begin{proposition} \label{prop:trop_mod_mink}
    Let $p : \wt\Sigma = \cT\cM_{\varphi}\Sigma \to \Sigma$ be a tropical modification along a tropical divisor $\Delta
    = \div(\varphi)$. Then $p_* : M_k(\wt\Sigma) \to M_k(\Sigma)$ is injective for all $k$.
\end{proposition}

\begin{proof}
    Given a cone $\sigma \in \Sigma_k$, the only $k$-dimensional cone $\tau$ of $\wt\Sigma$ such that $p(\tau) = \sigma$
    is $\tau = \wt\sigma$. It follows that for any $w \in M_k(\wt\Sigma)$ and any $\sigma \in \Sigma_k$, $p_*w(\sigma) =
    w(\wt\sigma)$.

    Suppose $p_*w=0$, i.e., $p_*w(\sigma)=w(\wt\sigma)=0$ for all $\sigma \in \Sigma_k$. The only other possible
    $k$-dimensional cones of $\wt\Sigma$ are $\tau_{\geq}$ for $\tau \in \Delta_{k-1}$, so we just need to show that
    $w(\tau_{\geq}) = 0$ for all such $\tau$. But given $\tau \in \Delta_{k-1}$, the $k$-dimensional cones of
    $\wt\Sigma$ containing $\wt\tau$ are precisely the cones $\wt\sigma$ for $\sigma \in \Sigma_k$, $\sigma \succ \tau$,
    and $\tau_{\geq}$. So the balancing condition of $w \in M_k(\wt\Sigma)$ at $\wt\tau \in \Sigma_{k-1}$ says
    \[
        \sum_{\substack{\sigma \in \Sigma_k \\ \sigma \succ \tau}} w(\wt\sigma)n_{\wt\sigma,\wt\tau} +
        w(\tau_{\geq})n_{\tau_{\geq},\wt\tau} = 0 \mod N_{\wt\tau}.
    \]
    Since by assumption $w(\wt\sigma) = 0$ for all $\sigma \in \Sigma_k$, this reduces to the equation
    \[
        w(\tau_{\geq})n_{\tau_{\geq},\wt\tau} = 0 \mod N_{\wt\tau}.
    \]
    By definition $n_{\tau_{\geq},\wt\tau} \neq 0 \mod N_{\wt\tau}$, so we conclude that $w(\tau_{\geq}) = 0$, hence
    $p_*$ is injective.
\end{proof}

\begin{corollary} \label{cor:irred_mod}
    Any tropical modification of an irreducible tropical fan is irreducible. Tropical modifications of locally
    irreducible tropical fans along trivial or locally irreducible divisors are locally irreducible.
\end{corollary}

\begin{proof}
    Suppose $p : (\wt\Sigma,\wt\omega) \to (\Sigma,\omega)$ is a tropical modification and $\Sigma$ is irreducible of
    dimension $d$. Then $p_* : M_d(\wt\Sigma) \to M_d(\Sigma) \cong \bZ \cdot \omega$ is injective by
    \cref{prop:trop_mod_mink}, and $p_*\wt\omega = \omega$ (cf. the proof of \cref{prop:trop_mod_mink}). Thus $p_* :
    M_d(\wt\Sigma) \to M_d(\Sigma)$ is an isomorphism, so $\wt\Sigma$ is irreducible by \cref{prop:irred_criterion}. The
    local case follows from \cref{prop:star_fan_mod}.
\end{proof}

\begin{proposition} \label{prop:trop_mod_chow}
    Let $p : \wt\Sigma = \cT\cM_{\varphi}\Sigma \to \Sigma$ be a tropical modification along a tropical divisor $\Delta
    = \div(\varphi)$. Then $p^* : A^*(\Sigma) \to A^*(\wt\Sigma)$ is a surjective morphism of graded rings, inducing a
    surjection $p^* : A^k(\Sigma) \to A^k(\wt\Sigma)$ for all $k$.
\end{proposition}

\begin{proof}
    In the case that $\Sigma$, hence $\wt\Sigma$, is unimodular, this is a direct verification using the presentation of
    the Chow ring of a unimodular fan in \cref{thm:chow_pres}. If the modification is degenerate, the result is
    immediate. Otherwise, $A^*(\wt\Sigma)$ is generated over $A^*(\Sigma)$ by the class $x_0$ corresponding to the
    unique additional ray $0_{\geq}$, and the lattice point $\wt m = (\vec{0},1) \in \wM \cong M \times \bZ$ gives the
    relation
    \[
      x_0 = -\sum_{\rho \in \Sigma_1} \langle \wt m, v_{\wt\rho} \rangle x_{\rho}.
    \]
    See \cite[Proposition 6.5]{aminiHomologyTropicalFans2021} for more details.

    If $\Sigma$ is not unimodular, let $\Sigma'$ be a unimodular refinement of $\Sigma$. This induces a unimodular
    refinement $\wt\Sigma'$ of $\wt\Sigma$ as well, allowing one to easily reduce to the unimodular case by using
    Kimura's description of the Chow ring of a singular variety in terms of a resolution
    \cite{shun-ichiFractionalIntersectionBivariant1992}, cf. \cite[Section 2]{payneEquivariantChowCohomology2006}.
\end{proof}

\begin{theorem} \label{thm:poincare_mod}
    Let $p : (\wt\Sigma,\wt\omega) \to (\Sigma,\omega)$ be a tropical modification.
    \begin{enumerate}
      \item If $- \cap \omega : A^k(\Sigma) \to M_{d-k}(\Sigma)$ is injective (resp. surjective) then $- \cap \wt\omega
          : A^k(\wt\Sigma) \to M_{d-k}(\wt\Sigma)$ is injective (resp. surjective).
      \item If $- \cap \omega: A^k(\Sigma) \to M_{d-k}(\Sigma)$ is injective, then $p^* : A^k(\Sigma) \to
          A^k(\wt\Sigma)$ is injective, hence an isomorphism by \cref{prop:trop_mod_chow}.
      \item If $- \cap \omega: A^k(\Sigma) \to M_{d-k}(\Sigma)$ is surjective, then $p_* : M_{d-k}(\wt\Sigma) \to
          M_{d-k}(\Sigma)$ is surjective, hence an isomorphism by \cref{prop:trop_mod_mink}.
    \end{enumerate}
    In particular if $(\Sigma,\omega)$ is a Poincar\'e tropical fan, then $(\wt\Sigma,\wt\omega)$ is a Poincar\'e
    tropical fan, and $p^* : A^*(\Sigma) \to A^*(\wt\Sigma)$ and $p_* : M_*(\wt\Sigma) \to M_*(\Sigma)$ are
    isomorphisms. Furthermore, a tropical modification of a star-Poincar\'e tropical fan along a trivial or
    star-Poincar\'e tropical divisor is star-Poincar\'e.
\end{theorem}

\begin{proof}
    Everything follows easily from \cref{prop:trop_mod_mink,prop:trop_mod_chow,prop:star_fan_mod} and the diagram
    \[
        \begin{tikzcd}
            A^k(\wt\Sigma) \ar[r, "- \cap \wt\omega"] & M_{d-k}(\wt\Sigma) \ar[d, "p_*"] \\
            A^k(\Sigma) \ar[u, "p^*"] \ar[r, "- \cap \omega"] & M_{d-k}(\Sigma),
        \end{tikzcd}
    \]
    which commutes by the projection formula.
\end{proof}

Recall that the properties of being reduced, locally irreducible, or star-Poincar\'e are intrinsic to the support,
local, and stably invariant \cref{prop:red_loc_irred_nice,cor:poincare_nice}. The above results show that these
properties are also preserved by tropical modifications, in the following sense.

\begin{definition}
    A property $\cP$ of tropical fans is \emph{preserved by tropical modifications} if whenever $\Sigma$ is a tropical
    fan satisfying $\cP$ and $\varphi$ is a piecewise integral linear function on $\Sigma$  such that $\div(\varphi)$ is
    either trivial or also satisfies $\cP$, then $\cT\cM_{\varphi}(\Sigma)$  satisfies $\cP$.
\end{definition}

\begin{corollary} \label{cor:mod_props}
    The properties of being reduced, locally irreducible, or star-Poincar\'e, are preserved by tropical modifications.
\end{corollary}

\begin{proof}
    This is \cref{lem:red_mod,cor:irred_mod,thm:poincare_mod}.
\end{proof}

\begin{proposition} \label{prop:trop_mod_product}
    Let $(\Sigma,\omega), (\Sigma',\omega')$ be two tropical fans. Suppose $(\Sigma,\omega)$ is reduced. Let $\varphi$
    be a piecewise integral linear function on $\Sigma'$ and define a piecewise integral linear function $\psi$ on
    $\Sigma \times \Sigma'$ by $\psi(x,y) = \varphi(y)$. Then
    \[
        \div(\psi) = \Sigma \times \div(\varphi) \;\; \text{ and } \;\; \cT\cM_{\psi}(\Sigma \times \Sigma') = \Sigma \times \cT\cM_{\varphi}\Sigma'.
    \]
\end{proposition}

\begin{proof}
    Say $\dim \Sigma = d$, $\dim \Sigma' = d'$. Then a codimension one cone of $\Sigma \times \Sigma'$ either looks like
    $\sigma \times \tau'$ for $\sigma \in \Sigma_d$, $\tau' \in \Sigma'_{d'-1}$, or $\tau \times \sigma'$ for $\tau \in
    \Sigma_{d-1}$, $\sigma' \in \Sigma'_{d'-1}$.

    In the former case, the top-dimensional cones containing $\sigma \times \tau'$ are all of the form $\sigma \times
    \sigma'$, where $\sigma'$ is a top-dimensional cone of $\Sigma'$ containing $\tau'$. Let $n_{\sigma \times
    \sigma',\sigma \times \tau'} = (v,n_{\sigma',\tau'})$, where $v$ is any lattice point in the relative interior of
    $\sigma$, and $n_{\sigma',\tau'}$ is any lattice point in the relative interior of $\sigma'$ whose image generates
    the one-dimensional quotient lattice $N_{\sigma'}/N_{\tau'}$. Then the image of $n_{\sigma \times \sigma', \sigma
    \times \tau'}$ generates the one-dimensional quotient lattice $N_{\sigma \times \sigma'}/N_{\sigma \times \tau'}
    \cong 0 \times N_{\sigma'}/N_{\tau'}$. Therefore,
    \begin{align*}
        \ord_{\sigma \times \tau'}(\psi) &= \psi_{\sigma \times \tau'}\left(\sum_{\substack{\sigma' \in \Sigma'_{d'} \\ \sigma' \succ \tau'}}
                                            \omega(\sigma)\omega(\sigma')(v,n_{\sigma',\tau'})\right) - \sum_{\substack{\sigma' \in \Sigma'_{d'} \\ \sigma' \succ \tau'}}
                                            \psi_{\sigma \times \sigma'}(\omega(\sigma)\omega(\sigma')(v,n_{\sigma',\tau'})) \\
                                        &= \varphi_{\tau'}\left(\sum_{\substack{\sigma' \in \Sigma'_{d'} \\ \sigma'
                                            \succ \tau'}} \omega(\sigma)\omega(\sigma')n_{\sigma',\tau'}\right) -
                                            \sum_{\substack{\sigma' \in \Sigma'_d \\ \sigma' \succ \tau'}}
                                            \varphi_{\sigma'}(\omega(\sigma)\omega(\sigma')n_{\sigma',\tau'}) \\
                                        &= \varphi_{\tau'}\left(\sum_{\substack{\sigma' \in \Sigma'_{d'} \\ \sigma'
                                            \succ \tau'}} \omega(\sigma')n_{\sigma',\tau'}\right) - \sum_{\substack{\sigma'
                                            \in \Sigma'_d \\ \sigma' \succ \tau'}}
                                            \varphi_{\sigma'}(\omega(\sigma')n_{\sigma',\tau'}) \\
                                        &= \ord_{\tau'}(\varphi).
    \end{align*}
    (The second-to-last equality is because $(\Sigma,\omega)$ is reduced, so $\omega(\sigma)=1$ for all $\sigma \in
    \Sigma_d$.)

    In the latter case, the top-dimensional cones of $\Sigma \times \Sigma'$ containing $\tau \times \sigma'$ are those
    of the form $\sigma \times \sigma'$ for $\sigma$ a top-dimensional cone of $\Sigma$ containing $\tau$.  Let
    $n_{\sigma \times \sigma', \tau \times \sigma'} = (n_{\sigma,\tau},v)$, where $n_{\sigma,\tau}$ is any lattice point
    in the relative interior of $\sigma$ whose image generates the one-dimensional quotient lattice
    $N_{\sigma}/N_{\tau}$, and $v$ is any lattice point in the relative interior of $\sigma'$. Then the image of
    $n_{\sigma \times \sigma', \tau \times \sigma'}$ generates the one-dimensional quotient lattice $N_{\sigma \times
    \sigma'}/N_{\tau \times \sigma'} \cong N_{\sigma}/N_{\tau} \times 0$.  Therefore,
    \begin{align*}
      \ord_{\tau \times \sigma'}(\psi) &= \psi_{\tau \times
      \sigma'}\left(\sum_{\substack{\sigma \in \Sigma_d \\ \sigma \succ \tau}}
    \omega(\sigma)\omega(\sigma')(n_{\sigma,\tau},v)\right) -
    \sum_{\substack{\sigma \in \Sigma_d \\ \sigma \succ \tau}} \psi_{\sigma
    \times \sigma'}(\omega(\sigma)\omega(\sigma')(n_{\sigma,\tau},v)) \\
                                      &= \varphi_{\sigma'}\left(\sum_{\substack{\sigma \in
                                        \Sigma_d \\ \sigma \succ \tau}}
                                        \omega(\sigma)\omega(\sigma')v\right)
                                        - \sum_{\substack{\sigma \in \Sigma_d \\
                                          \sigma \succ \tau}} \varphi_{\sigma'}(\omega(\sigma)\omega(\sigma')v) \\
                                      &= 0
    \end{align*}
    by linearity of $\varphi_{\sigma'}$. It follows that $\div(\psi) = \Sigma \times \div(\varphi)$.

    Now if $\sigma \times \sigma'$ is any cone of $\Sigma \times \Sigma'$, then
    \begin{align*}
      \wt{\sigma \times \sigma'} &= \{(x,y,\psi(x,y)) \mid (x,y) \in \sigma \times \sigma'\} \\
                                 &= \{(x,y,\varphi(y)) \mid (x,y) \in \sigma \times \sigma'\} \\
                                 &= \sigma \times \wt{\sigma'}.
    \end{align*}
    If $\sigma \times \tau \in \div(\psi) = \Sigma \times \div(\varphi)$, then
    \begin{align*}
      (\sigma \times \tau)_{\geq} &= \wt{\sigma \times \tau} + \bR_{\geq 0}(0,0,1) \\
                                  &= \sigma \times \wt\tau + \bR_{\geq 0}(0,0,1) \\
                                  &= \sigma \times \tau_{\geq}.
    \end{align*}
    It follows that $\cT\cM_{\psi}(\Sigma \times \Sigma') = \Sigma \times \cT\cM_{\varphi}\Sigma'$.
\end{proof}

It will be useful to us to speak of tropical modifications without referring to a specific tropical fan structure.

\begin{definition}
    Let $\cF$ be a tropical fan cycle. A \emph{piecewise integral linear function} on $\cF$ is a continuous function
    $\varphi : \cF \to \bR$ which is piecewise integral linear for some fan structure $\Sigma$ on $\cF$. The
    \emph{tropical modification} $\cT\cM_{\varphi}\cF$ of $\cF$ with respect to $\varphi$ is the tropical fan cycle in
    $\wN_{\bR} = N_{\bR} \times \bR$ defined by $\cT\cM_{\varphi}(\Sigma)$.
\end{definition}

\section{Quasilinear fans} \label{sec:qlin_fans}

\begin{definition} \label{def:qlin_fans}
    A reduced tropical fan cycle $\cF$ is \emph{quasilinear} if it is isomorphic to a complete tropical fan cycle, or a
    tropical modification of a quasilinear tropical fan cycle along a quasilinear tropical divisor.

    A tropical fan $\Sigma$ is \emph{quasilinear} if it is supported on a quasilinear tropical fan cycle.
\end{definition}

\begin{remark}
    The definition of quasilinearity is inductive on $\rk N_{\cF}$.
\end{remark}

\begin{theorem} \label{thm:qlin_poincare}
    Quasilinear tropical fans are reduced, locally irreducible, and star-Poincar\'e.
\end{theorem}

\begin{proof}
    Recall from \cref{cor:mod_props,prop:red_loc_irred_nice,cor:poincare_nice} that these properties are intrinsic to
    the support, local, stably invariant, and preserved by tropical modifications. The result follows by the definition
    of quasilinearity.
\end{proof}

\begin{corollary}
    Let $\wt\Sigma$ be a degenerate tropical modification of a quasilinear tropical fan $\Sigma$. Then $\wt\Sigma \cong
    \Sigma$, and in particular $\wt\Sigma$ is also quasilinear.
\end{corollary}

\begin{proof}
    Since quasilinear tropical fans are star-Poincar\'e, this follows from the discussion of \cref{rmk:degen_mods}.
\end{proof}

\begin{theorem} \label{thm:prod_qlin}
    Let $\cF$ and $\cG$ be two reduced tropical fan cycles. Then $\cF \times \cG$ is quasilinear if and only if $\cF$
    and $\cG$ are both quasilinear.
\end{theorem}

\begin{proof}
    Note that $\cF \times \cG$ is complete $\iff$ $\cF$ and $\cG$ are both complete. Therefore, we can assume without
    loss of generality that $\cG$ and $\cF \times \cG$ are not complete. We will prove the result by induction on the
    rank $n$ of the ambient lattice of $\cF \times \cG$. Note the base case $n\leq 2$ is trivial.

    ($\impliedby$) Suppose $\cF$ and $\cG$ are both quasilinear. By assumption $\cG$ is a non-complete quasilinear
    tropical fan cycle, so $\cG \cong \cT\cM_{\cD}\cG'$ for some quasilinear tropical fan cycle $\cG'$ and quasilinear
    tropical divisor $\cD$ on $\cG'$. Then by \cref{prop:trop_mod_product}
    \[
      \cF \times \cG \cong \cF \times \cT\cM_{\cD}\cG' \cong \cT\cM_{\cF \times \cD}(\cF \times \cG'),
    \]
    so by induction $\cF \times \cG$ is quasilinear.

    ($\implies$) Suppose $\cF \times \cG$ is quasilinear. By assumption both $\cF \times \cG$ and $\cG$ are
    non-complete, so (after composing with an isomorphism if necessary) we can assume that the projection to the last
    coordinate realizes $\cF \times \cG$ as a tropical modification of a quasilinear tropical fan of the form $\cF
    \times \cG'$ along a quasilinear tropical divisor $\cD$. By induction $\cG'$ is quasilinear. Since by
    \cref{prop:chow_prods} we have
    \[
        M_{d + d'-1}(\cF \times \cG') = (M_d(\cF) \otimes M_{d'-1}(\cG')) \oplus (M_{d-1}(\cF) \otimes M_{d'}(\cG')),
    \]
    it follows that $\cD = \cF \times \cD' + \cD'' \times \cG'$ for some (possibly trivial) divisors $\cD'$ on $\cG'$
    and $\cD''$ on $\cF$. Since $\cD$ is quasilinear, it is irreducible. Both $\cF \times \cD'$ and $\cD'' \times \cG'$
    are contained in $\cD$ and if nontrivial then they have the same dimension as $\cD$. It follows by irreducibility
    that one of them must be trivial and the other must be $\cD$ itself. Since $\cF \times \cG$ looks like the tropical
    modification of $\cF \times \cG'$ along $\cD$, the only possibility is that $\cD = \cF \times \cD'$ and $\cD''
    \times \cG'$ is trivial. So we conclude that
    \[
        \cF \times \cG \cong \cT\cM_{\cF \times \cD'}(\cF \times \cG') \cong \cF \times \cT\cM_{\cD'}\cG',
    \]
    hence $\cG = \cT\cM_{\cD'}\cG'$. Since both $\cD'$ and $\cG'$ are quasilinear by induction, it follows that $\cG$ is
    quasilinear.
\end{proof}

\begin{theorem} \label{thm:star_qlin}
    Star fans of quasilinear tropical fans are quasilinear.
\end{theorem}

\begin{proof}
    Since star fans of complete fans are complete, and the star fans of a tropical modification
    $\cT\cM_{\Delta}(\Sigma)$ look either like (possibly degenerate) modifications of star fans of $\Sigma$ along star
    fans of $\Delta$, or star fans of $\Sigma$ (\cref{prop:star_fan_mod}), it follows by induction on the rank of the
    ambient lattice that if star fans of $\Sigma$ and $\Delta$ are quasilinear, then so are star fans of
    $\cT\cM_{\Delta}(\Sigma)$. The result now follows by \cref{thm:prod_qlin} and the discussion following
    \cref{def:nice_props}.
\end{proof}

\begin{corollary}
    The property of being a quasilinear tropical fan is intrinsic to the support, local, stably invariant, and preserved
    by tropical modifications.
\end{corollary}

\begin{proof}
    ``Intrinsic to the support'' and ``preserved by tropical modifications'' are by definition, and ``local'' and
    ``stably invariant'' are by \cref{thm:star_qlin,thm:prod_qlin}.
\end{proof}

\subsection{Shellable fans}

\begin{definition}[{\cite[Definitions 5.3, 5.4]{aminiHomologyTropicalFans2021}}]
    The class of \emph{(tropically) shellable} tropical fans is the smallest class of reduced tropical fans containing
    the $0$-fan and the (unique) complete fan in $\bR^1$, and which is closed under products, tropical modifications
    along trivial or shellable fans, and refinements and coarsenings along rays in the relative interior of cones whose
    star fans are shellable.
\end{definition}

The motivation for this terminology comes from the observation that the simplicial complex associated to a tropically
shellable tropical fan is shellable in the usual sense.

\begin{theorem} \label{thm:qlin_shellable}
    A reduced tropical fan is quasilinear if and only if it is shellable.
\end{theorem}

\begin{proof}
    ($\implies$) The class of quasilinear tropical fans includes the $0$ fan and the unique complete fan in $\bR^1$, is
    closed under products by \cref{thm:prod_qlin}, and is closed under refinements/coarsenings and tropical
    modifications by definition. Therefore, any reduced quasilinear fan is shellable.

    ($\impliedby$) Let $\Sigma$ be a reduced tropical fan. We will prove by induction on $n = \rk N_{\Sigma}$  that
    $\Sigma$ is shellable $\implies$ $\Sigma$ is quasilinear. In the case $n=1$ the result is immediate, as the
    shellable tropical fans in $\bR^1$ are precisely the (unique) complete fan and the $0$ fan, and these are also
    precisely the quasilinear tropical fans in $\bR^1$. Now assume the result for all $k < n$. Let $\Sigma$ be a
    shellable tropical fan in $\bR^n$. Then $\Sigma$ is supported on some fan which is obtained either by a tropical
    modification of a shellable tropical fan in $\bR^{n-1}$ along a trivial or shellable tropical divisor, or as the
    product of two shellable tropical fans in smaller ambient spaces. By induction, all of these smaller fans are
    quasilinear. Since products and tropical modifications of quasilinear fans are quasilinear, it follows that $\Sigma$
    is quasilinear.
\end{proof}

\begin{corollary}
    Quasilinear tropical fans are smooth in the sense of \cite{aminiHomologyTropicalFans2021}, and simplicial
    quasilinear tropical fans satisfy hard Lefschetz and the Hodge-Riemann relations.
\end{corollary}

\begin{proof}
    This follows from the above theorem and \cite[Theorem 5.7]{aminiHomologyTropicalFans2021}.
\end{proof}

\subsection{Examples} \label{sec:example_fans}

Recall that by a \emph{linear fan} we mean any (reduced) tropical fan supported on the Bergman fan of a matroid.

\begin{theorem} \label{thm:lin_qlin}
    Linear fans are quasilinear.
\end{theorem}

\begin{proof}
    Let $\Sigma_M$ be the Bergman fan associated to the matroid $M$. Then either $\Sigma_M$ is complete, or $\Sigma_M$
    is a tropical modification of $\Sigma_{M \setminus i}$ along $\Sigma_{M/i}$ \cite[Proposition
    2.25]{shawTropicalIntersectionProduct2013}. The result follows by induction.
\end{proof}

\begin{example} \label{ex:quasilinear_fans}
    We classify (up to isomorphism) all quasilinear tropical fan cycles in $\bR^n$ for $n \leq 3$. Note the trivial fan
    is quasilinear, and complete fans are quasilinear. These are all of the possible fans in $\bR^n$ for $n \leq 1$. For
    $n=2,3$, we only have to consider nontrivial, non-complete fans.
    \begin{enumerate}
        \item There are (up to isomorphism) two one-dimensional quasilinear fans in $\bR^2$.  They are pictured in
            \cref{fig:lines_R2}. We call the first fan the ``classical line'' and the second fan the ``(standard)
            tropical line'' (in $\bR^2$).  The trivial line is a degenerate modification of $\bR^1$; the tropical line
            is a modification of $\bR^1$ along $0$.
            \begin{figure}[htpb]
            \begin{center}
            \begin{subfigure}[b]{0.4\textwidth}
              \centering
              \begin{tikzpicture}[scale=1, transform shape]
                \draw[thick,->] (0,0) -- (1,1);
                \draw[thick,->] (0,0) -- (-1,-1);
                \filldraw[black] (0,0) circle (2pt);
              \end{tikzpicture}
              \caption{The classical line.}%
            \end{subfigure}
            \begin{subfigure}[b]{0.4\textwidth}
              \centering
              \begin{tikzpicture}[scale=1, transform shape]
                \draw[thick,->] (0,0) -- (1.4,0);
                \draw[thick,->] (0,0) -- (-1,-1);
                \draw[thick,->] (0,0) -- (0,1.4);
              \end{tikzpicture}
              \caption{The standard tropical line.}%
              \end{subfigure}
            \end{center}
            \caption{One-dimensional quasilinear fans in $\bR^2$.}%
            \label{fig:lines_R2}
            \end{figure}
        \item There are three types of one-dimensional quasilinear fans in $\bR^3$. The first is a classical line $\bR^1
            \subset \bR^3$.

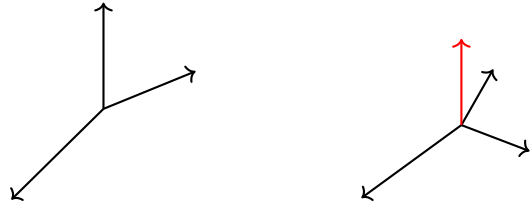
            The seecond is the fan shown in \cref{fig:mixed_line_R3}; it can be viewed either as a degenerate
            modification of the standard tropical line in $\bR^2$, or as a nondegenerate modification of the classical
            line in $\bR^2$.

            The third is the fan shown in \cref{fig:trop_line_R3}, known as the ``standard tropical line in $\bR^3$.''
            \begin{figure}[htpb]
            \begin{center}
            \begin{subfigure}[b]{0.4\textwidth}
              \centering
              \tdplotsetmaincoords{45}{15}
              \begin{tikzpicture}[scale=1, transform shape, tdplot_main_coords]
                \draw[thick,->] (0,0,0) -- (1,1,0);
                \draw[thick,->] (0,0,0) -- (-1,-1,-1);
                \draw[thick,->] (0,0,0) -- (0,0,2);
              \end{tikzpicture}
              \caption{A degenerate modification of the standard tropical line in $\bR^2$}%
              \label{fig:mixed_line_R3}
            \end{subfigure}
            \begin{subfigure}[b]{0.4\textwidth}
            \tdplotsetmaincoords{35}{25}
            \begin{tikzpicture}[scale=1, transform shape, tdplot_main_coords]
              \draw[thick,->] (0,0,0) -- (1,0,0);
              \draw[thick,->] (0,0,0) -- (0,1,0);
              \draw[thick,->] (0,0,0) -- (-1,-1,-1);
              \draw[thick, ->, red] (0,0,0,) -- (0,0,2);
            \end{tikzpicture}
            \caption{A nondegenerate modification of the standard tropical line in $\bR^2$}%
            \label{fig:trop_line_R3}
            \end{subfigure}
            \end{center}
            \end{figure}
        \item There are three types of two-dimensional quasilinear fans in $\bR^3$. The first is a classical plane
            $\bR^2 \subset \bR^3$, which can be viewed as a degenerate tropical modification of $\bR^2$.

          The second is a tropical modification of $\bR^2$ along a classical line, pictured in \cref{fig:trop_modR2_x}
          It is also isomorphic to the product of the standard tropical line in $\bR^2$ with $\bR^1$.

          The third is a tropical modification of $\bR^2$ along the standard tropical line, pictured in
          \cref{fig:trop_modR2_xy0}. This fan is known as the ``standard tropical plane (in $\bR^3$).''
          \begin{figure}[htpb]
          \begin{center}
          \begin{subfigure}[b]{0.4\textwidth}
            \centering
            \tdplotsetmaincoords{45}{15}
            \begin{tikzpicture}[scale=1, transform shape, tdplot_main_coords]
              \draw[thick] (0,0,0) -- (1,0,0);
              \draw[thick] (0,0,0) -- (-1,0,-1);

              \draw[thick, red] (0,0,0) -- (0,1,0);
              \draw[thick, red] (0,0,0) -- (0,-1,0);
              \draw[thick, red] (0,0,0) -- (0,0,1);
              \draw[thick,red] (0,1,0) -- (0,1,1);
              \draw[thick,red] (0,-1,0) -- (0,-1,1);
              \draw[thick,red] (0,1,1) -- (0,0,1);
              \draw[thick,red] (0,-1,1) -- (0,0,1);

              \draw[thick] (0,1,0) -- (1,1,0);
              \draw[thick] (0,-1,0) -- (1,-1,0);
              \draw[thick] (1,0,0) -- (1,1,0);
              \draw[thick] (1,0,0) -- (1,-1,0);

              \draw[thick] (0,-1,0) -- (-1,-1,-1);
              \draw[thick] (0,1,0) -- (-1,1,-1);
              \draw[thick] (-1,0,-1) -- (-1,-1,-1);
              \draw[thick] (-1,0,-1) -- (-1,1,-1);
            \end{tikzpicture}
            \caption{The tropical modification of $\bR^2$ with respect to
            $\varphi'(x,y)=\min\{x,0\}$.}%
            \label{fig:trop_modR2_x}
          \end{subfigure}
          \begin{subfigure}[b]{0.4\textwidth}
          \tdplotsetmaincoords{35}{25}
          \begin{tikzpicture}[scale=1, transform shape, tdplot_main_coords]
            \draw[thick, red] (0,0,0) -- (1,0,0);
            \draw[thick, red] (0,0,0) -- (0,1,0);
            \draw[thick, red] (0,0,0) -- (-1,-1,-1);

            \draw[thick, red] (0,0,0) -- (0,0,1);
            \draw[thick, red] (1,0,0) -- (1,0,1);
            \draw[thick, red] (0,1,0) -- (0,1,1);
            \draw[thick, red] (-1,-1,-1) -- (-1,-1,0);
            \draw[thick, red] (0,0,1) -- (1,0,1);
            \draw[thick, red] (0,0,1) -- (0,1,1);
            \draw[thick, red] (0,0,1) -- (-1,-1,0);

            \draw[thick] (1,0,0) -- (1,1,0);
            \draw[thick] (0,1,0) -- (1,1,0);
            \draw[thick] (1,0,0) -- (0,-1,-1);
            \draw[thick] (-1,-1,-1) -- (0,-1,-1);
            \draw[thick] (0,1,0) -- (-1,0,-1);
            \draw[thick] (-1,-1,-1) -- (-1,0,-1);
          \end{tikzpicture}
          \caption{The tropical modification of $\bR^2$ with respect to
          $\varphi(x,y)=\min\{x,y,0\}$.}%
          \label{fig:trop_modR2_xy0}
          \end{subfigure}
          \end{center}
          \end{figure}
    \end{enumerate}
\end{example}

\begin{example}
    The line $2y=3x$ in $\bR^2$ is isomorphic to $\bR^1$ via the linear map $\bR^2 \to \bR^1$, $(x,y) \mapsto y-x$. Thus
    this line is a quasilinear tropical fan cycle, even though neither coordinate projection $\bR^2 \to \bR^1$ realizes
    it as a tropical modification.
\end{example}

\begin{example} \label{ex:arb_line}
    More generally, let $a,b \in \bZ \setminus \{0\}$ and let $\Sigma$ be the line $ax=by$ in $\bR^2$. Let $d =
    \gcd(a,b)$.  Then there exist integers $x,y$ such that $ax + by = d$, and so the linear map
    \[
        N_{\Sigma} \to \bZ, \;\; (a,b) \mapsto \frac{a}{d}x + \frac{b}{d}y = 1
    \]
    induces an isomorphism of $\Sigma$ with $\bR^1$. Thus $\Sigma$ is quasilinear.
\end{example}

\begin{example}
    The fan $\Sigma$ in $\bR^2$ pictured in \cref{fig:cross_R2} is not irreducible, and therefore not quasilinear. On the
    other hand, if $\varphi(x,y) = x-y$ for $x+y \geq 0$ and $0$ otherwise, then the tropical modification of $\Sigma$
    with respect to $\varphi$ is isomorphic to the standard tropical line in $\bR^3$ \cite[Example
    11.4]{aminiHomologyTropicalFans2021}. This is a nontrivial degenerate modification of $\Sigma$.
    \begin{figure}[htpb]
    \begin{center}
    \centering
    \begin{tikzpicture}[scale=1, transform shape]
      \draw[thick,->] (0,0) -- (1,0);
      \draw[thick,->] (0,0) -- (0,1);
      \draw[thick,->] (0,0) -- (-1,0);
      \draw[thick,->] (0,0) -- (0,-1);
      \filldraw[black] (0,0) circle (2pt);
    \end{tikzpicture}
    \end{center}
    \caption{A reducible fan in $\bR^2$.}%
    \label{fig:cross_R2}
    \end{figure}
\end{example}

\begin{example}
    The fan $\Sigma$ in $\bR^2$ pictured in \cref{fig:nonreduced_R2} is tropical with weights $1,2,1$ on the rays through
    $(1,0)$, $(0,1)$, $(-1,-2)$, respectively. In particular this fan is not reduced, and therefore not quasilinear.
    Note $\Sigma$ is the tropical modification of $\bR^1$ with respect to the function $\varphi(x) = \min\{2x,0\}$; the
    divisor of this function is the origin with weight $2$.
    \begin{figure}[htpb]
    \begin{center}
    \centering
    \begin{tikzpicture}[scale=1, transform shape]
      \draw[thick,->] (0,0) -- (1.7,0);
      \draw[thick,->] (0,0) -- (0,1.7);
      \draw[thick,->] (0,0) -- (-1,-2);
      \filldraw[black] (0,0) circle (2pt);
    \end{tikzpicture}
    \end{center}
    \caption{A nonreduced fan in $\bR^2$.}%
    \label{fig:nonreduced_R2}
    \end{figure}
\end{example}

\section{Quasilinear varieties} \label{sec:qlin_vars}

\subsection{Preliminaries}

We will assume some familiarity with the basics of tropicalizations and tropical compactifications of closed
subvarieties of tori, see for instance \cite{tevelevCompactificationsSubvarietiesTori2007,
luxtonResultsTropicalCompactifications2011, gublerGuideTropicalizations2013}, \cite[Chapter
6]{maclaganIntroductionTropicalGeometry2015}. We quickly summarize here the relevant definitions and facts.

Let $T$ be an algebraic torus and $X(\Sigma)$ a toric variety with torus $T$.  Recall that the closure $\oY \subset
X(\Sigma)$ of a pure-dimensional closed subscheme $Y \subset T$ is a \emph{tropical compactification} if $\oY$ is proper
and the multiplication map $m : \oY \times T \to X(\Sigma)$ is faithfully flat (i.e., flat and surjective)
\cite[Definition 1.1]{tevelevCompactificationsSubvarietiesTori2007}. This implies in particular that $\oY$ intersects
each torus orbit $O(\sigma)$ of $X(\Sigma)$ nontrivially in the expected dimension, inducing a stratification of $\oY$
in dimension $k$ by strata of the form $Y_{\sigma} = \oY \cap O(\sigma)$ for $\sigma \in \Sigma_k$. The tropical
compactification $\oY \subset X(\Sigma)$ is called \emph{sch\"on} if the multiplication map $m : \oY \times T \to
X(\Sigma)$ is in addition smooth \cite[Definition 1.3]{tevelevCompactificationsSubvarietiesTori2007}, or equivalently
all strata $Y_{\sigma}$ of $\oY$ are smooth \cite[Lemma 2.7]{hackingHomologyTropicalVarieties2008}.

The tropicalization\footnote{This is not the way tropicalization is usually defined, but nevertheless it agrees with the
usual definition---see \cite[Chapter 6]{maclaganIntroductionTropicalGeometry2015}.} $\trop(Y)$ of $Y$ is the tropical
fan cycle defined by the fan $\Sigma$ of any tropical compactification $\oY \subset X(\Sigma)$ of $Y$, where the weight
$\omega(\sigma)$ of a top-dimensional cone $\sigma$ of $\Sigma$ is equal to the length of the zero-dimensional scheme
$Y_{\sigma}$. We refer to $\omega$ as the \emph{induced weight (from $Y$)} when we wish to emphasize that it is the
weight coming from $\trop(Y)$, rather than some other weight on $\trop(Y)$ as an abstract tropical fan cycle.

While it is not the case that every tropical fan supported on $\trop(Y)$ yields a tropical compactification of $Y$
\cite[Example 3.10]{sturmfelsEliminationTheoryTropical2008}, it is true that if $\Sigma$ gives a tropical
compactification of $Y$, then so does any refinement of $\Sigma$ \cite[Proposition
2.5]{tevelevCompactificationsSubvarietiesTori2007}; furthermore, if $Y$ admits a sch\"on tropical compactification, then
\emph{any} fan supported on $\trop(Y)$ also gives a sch\"on tropical compactification of $Y$, hence it makes sense to
call $Y$ itself sch\"on \cite[Theorem 1.5]{luxtonResultsTropicalCompactifications2011}.

\begin{lemma} \label{lem:trop_integral}
    If $Y \subset T$ is any pure-dimensional closed subscheme such that $\trop(Y)$ is irreducible and the induced weight
    is the fundamental weight, then $Y$ is irreducible and generically reduced. In particular, if $\trop(Y)$ is reduced
    and irreducible, then $Y$ is irreducible and generically reduced.
\end{lemma}

\begin{proof}
    Let $Y = Y_1 \cup \cdots \cup Y_k$ be the decomposition of $Y$ into its irreducible components, with multiplicities
    $a_1,\ldots,a_k$. Then $\trop(Y) = \trop(Y_1) \cup \cdots \cup \trop(Y_k)$ by \cite[Definition 2.5.2, Proposition
    2.7.2]{ossermanLiftingTropicalIntersections2013}. Each $\trop(Y_i)$ is a tropical fan cycle of the same dimension as
    $\trop(Y)$, with support contained in $\trop(Y)$, so by irreducibility of $\trop(Y)$ we have
    \[
      \trop(Y) = \trop(Y_1) = \cdots = \trop(Y_k)
    \]
    as sets. Furthermore,
    \[
      \omega = a_1\omega_1 + \cdots + a_k\omega_k,
    \]
    where $\omega$, $\omega_i$ are the induced weights on $\trop(Y)$, $\trop(Y_i)$. Since $\trop(Y_i)=\trop(Y)$ for all
    $i$, we see that $\omega_i(\sigma) \neq 0$ for every cone $\sigma$ in $\trop(Y)$. Since $\trop(Y)$ is irreducible
    and $\omega$ is the fundamental weight on $\trop(Y)$, it follows by \cref{prop:irred_criterion} that for
    $i=1,\ldots,k$ there is some integer $b_i > 0$ such that $\omega_i = b_i\omega$. Then
    \[
      \omega = (a_1b_1 + \cdots + a_kb_k)\omega,
    \]
    hence
    \[
      1 = a_1b_1 + \cdots + a_kb_k.
    \]
    Since the $a_i$'s and $b_i$'s are all positive integers, it follows that $k=1$ and $a_1=b_1=1$, so $Y$ is
    irreducible and generically reduced. The last sentence follows from the first, since the statement that $\trop(Y)$
    is reduced and irreducible means that $\trop(Y)$ is irreducible and reduced with the induced weight, thus the
    induced weight is the fundamental weight.
\end{proof}

\begin{lemma} \label{lem:trop_strata}
    Let $\oY \subset X(\Sigma)$ be a tropical compactification of $Y \subset T$.  Then for every $\sigma \in \Sigma$,
    $Y_{\sigma} = \oY \cap O(\sigma)$ is nonempty of pure dimension $\dim \oY - \dim \sigma$, and $\trop(Y_{\sigma}
    \subset O(\sigma))$ is a tropical fan cycle supported on $\Sigma^{\sigma}$, where the weight on $\Sigma^{\sigma}$
    induced by $Y_{\sigma}$ is the same as the weight induced by $\Sigma$.
\end{lemma}

\begin{proof}
    See \cite[Proposition 13.13, 14.3, 14.7]{gublerGuideTropicalizations2013}.
\end{proof}

\begin{remark} \label{rmk:in_degens}
    The tropicalization of $Y$ can equivalently be defined as the closure in $N_{\bR}$ of the set of points $w \in N$
    such that the initial degeneration $\text{in}_w(Y)$ is nonempty, see for instance
    \cite{gublerGuideTropicalizations2013}. If $\Sigma$ is a fan structure on $\trop(Y)$ giving a tropical
    compactification $\oY \subset X(\Sigma)$, and $w$ is in the relative interior of a cone $\sigma \in \Sigma$, then
    $\text{in}_w(Y) \cong Y_{\sigma} \times (\bC^*)^{\dim \sigma}$ \cite[Lemma
    3.6]{helmMonodromyFiltrationsTopology2012}, reflecting the fact that $\trop(Y)^{w} = \lvert \Sigma^{\sigma} \rvert
    \times \bR^{\dim \sigma}$ (cf. \cref{ex:local_defs}).
\end{remark}

\subsection{Quasilinear varieties}

\begin{definition}
  A closed subvariety $Y$ of an algebraic torus $T$ is called \emph{quasilinear} if $\trop(Y)$ is a quasilinear tropical
  fan cycle.
\end{definition}

\begin{remark}
    We emphasize in particular that this definition implies $\trop(Y)$ is irreducible, and reduced with the induced
    weight from $Y$, cf. \cref{thm:qlin_poincare}.
\end{remark}

For examples of quasilinear varieties, see \cref{sec:examples_vars,sec:applications}.

\begin{lemma} \label{lem:mod_bir}
    Fix an algebraic torus $\wT$ with cocharacter lattice $\wN$.  Let $\wY \subset \wT$ be an irreducible closed
    subvariety such that $\trop(\wY)$ is irreducible and the induced weight is the fundamental weight. If there is a
    linear map $p : \wN_{\bR} \to N_{\bR}$, $\rk N = \rk \wN - 1$, realizing $\trop(\wY)$ as a tropical modification of
    an irreducible tropical fan cycle $\cF$ in $N_{\bR}$, then the corresponding projection of tori $\pi : \wT \to T$
    restricts to a birational morphism with finite fibers $\pi : \wY \to Y$ to an irreducible closed subvariety $Y
    \subset T$ such that $\trop(Y) = \cF$ as tropical fan cycles.
\end{lemma}

\begin{proof}
    Let $Y = \overline{\pi(\wY)}$. Then $Y$ is an irreducible closed subvariety of $T$, and $\pi : \wY \to Y$ is
    dominant, so by \cite[Proposition 2.8]{sturmfelsEliminationTheoryTropical2008},
    \[
        \trop(Y) = p(\trop(\wY)) = \cF.
    \]
    In particular,
    \[
        \dim Y = \dim \trop(Y) = \dim \trop(\wY) = \dim \wY,
    \]
    so $\pi : \wY \to Y$ is a dominant morphism of irreducible varieties of the same dimension, hence is generically
    finite, say of degree $\delta$. Let $\wt\omega$ and $\omega$ be the induced weights on $\trop(\wY)$ and $\trop(Y)$.
    By \cite[Theorem 3.12]{sturmfelsEliminationTheoryTropical2008},
    \[
        \omega = \frac{1}{\delta}p_*\wt\omega.
    \]
    But since $\wt\omega$ is the fundamental weight on $\trop(\wY)$, it pushes forward to the fundamental weight on
    $\trop(Y)$ by \cref{cor:irred_mod}. Thus $\omega=p_*\wt\omega$ and $\delta=1$. Therefore, $\pi : \wY \to Y$ is
    dominant and generically finite of degree one, i.e., it is birational.

    Let $y \in Y$. Then $\pi^{-1}(y) \subset \wY$ is a closed subscheme of $\wY$ of dimension $\leq 1$. Suppose
    $\pi^{-1}(y)$ is nonempty. Then $\pi : \pi^{-1}(y) \to y$ is dominant, so by \cite[Proposition
    2.8]{sturmfelsEliminationTheoryTropical2008}, $p(\trop(\pi^{-1}(y)))=\trop(y) = 0 \in N_{\bR}$. Thus if $\wt\Sigma$ is a fan
    structure on $\trop(\wY)$ inducing a fan structure on $\trop(\pi^{-1}(y))$, then a cone $\sigma \in \wt\Sigma$ lies
    in $\trop(\pi^{-1}(y)) \iff p(\sigma)=0$. Since $p$ realizes $\wt\Sigma$ as a tropical modification of a fan
    structure on $\cF$, we see that the only possible cones of $\trop(\pi^{-1}(y))$ are the $0$-cone, and $0_{\geq}$.
    But there is no way for the one-dimensional fan consisting of just the ray $0_{\geq}$ to be balanced, so we conclude
    that $\trop(\pi^{-1}(y))=0$, and in particular $\dim \pi^{-1}(y) = \dim \trop(\pi^{-1}(y)) = 0$. Thus every fiber of
    $\pi : \wY \to Y$ is either empty or zero-dimensional.
\end{proof}

We introduce two general properties of varieties which will be useful later (\cref{sec:chow}).

\begin{definition} \label{def:chow_free}
    An irreducible variety $Y$ is \emph{Chow-free} if $A_{\dim Y}(Y) = \bZ \cdot [Y]$ and $A_k(Y) = 0$ for $k \neq \dim
    Y$.
\end{definition}

\begin{lemma}
    Any open subvariety of a Chow-free variety is Chow-free.
\end{lemma}

\begin{proof}
    Let $Y$ be Chow-free, $U \subset Y$ any (nonempty) open subvariety, and $Z = Y \setminus U$. Then $\dim Z < \dim Y$,
    so the result is immediate from the standard exact sequence
    \[
        A_k(Z) \to A_k(Y) \to A_k(U) \to 0.
    \]
\end{proof}

\begin{example}
    Affine space is Chow-free, so any open subvariety of affine space is Chow-free. In particular, any open subvariety
    of an algebraic torus is Chow-free.
\end{example}

\begin{definition} \label{def:weakly_linear}
    A variety $Y$ is \emph{linearly stratified} if it is isomorphic to an affine space or the complement of a linearly
    stratified variety in a linearly stratified variety, or if it contains a linearly stratified variety $Z$ such that
    the complement $Y \setminus Z$ is also linearly stratified.
\end{definition}

\begin{remark}
    Linearly stratified varieties are called ``linear'' in \cite{totaroChowGroupsChow2014,
    jannsenMixedMotivesAlgebraic2006}.  We rename them ``linearly stratified'' to avoid confusion with our notion of
    linear varieties.

    It follows from the definition that a variety stratified by linearly stratified varieties is itself linearly
    stratified. Thus the class of linearly stratified varieties is strictly larger than both the class of linear varieties (in our sense) and
    the class of quasilinear varieties.
\end{remark}

Recall from \cref{sec:morphisms} that if $\cF$ is a tropical fan cycle in $N_{\bR}$, then $N_{\cF} = N_{\cF,\bR} \cap
N$, where $N_{\cF,\bR}$ is the vector subspace of $N_{\bR}$ spanned by $\cF$.

\begin{definition}
    Let $Y \subset T$ be a closed subvariety. The \emph{minimal torus} of $Y$ is the torus $T_Y = N_{\trop Y} \otimes
    \bC^* \subset T$.
\end{definition}

\begin{definition}
    Let $Y \subset T$ and $Y' \subset T'$ be closed subvarieties of algebraic tori $T$ and $T'$ with respect character
    lattices $N$ and $N'$. A \emph{tropical isomorphism} $Y \xrightarrow{\sim} Y'$ is a linear isomorphism $N_{\trop Y}
    \xrightarrow{\sim} N_{\trop Y'}$, inducing an isomorphism $\trop(Y) \xrightarrow{\sim} \trop(Y')$, and such that the
    corresponding isomorphism of minimal tori $T_Y \xrightarrow{\sim} T'_{Y'}$ restricts to an isomorphism $Y
    \xrightarrow{\sim} Y'$.
\end{definition}

\begin{remark} \label{rmk:trop_iso}
    By \cite[Proposition 5.3]{jellConstructingSmoothFully2020}, if $Y \subset T$ and $T_Y$ is the minimal torus of $Y$,
    then some translate $Y' \cong Y$ of $Y$ is contained in $T_Y$, and $\trop(Y' \subset T_Y) \cong \trop(Y \subset T)$.
    Thus $Y \subset T$ is tropically isomorphic to $Y' \subset T_Y$, so there is no harm in replacing the original
    embedding $Y \subset T$ with $Y' \subset T_Y$. For the rest of this section we will therefore assume that $T$ is
    itself the minimal torus of $Y$.
\end{remark}

\begin{definition}
    Let $Y \subset T$ be a closed subvariety of an algebraic torus $T$, and $f$ a regular function on $Y$. The
    \emph{very affine graph} of $f$ on $Y$ is the closed subvariety $\Gamma^a_f = \Gamma_f \cap (Y \times \bC^*) \subset
    T \times \bC^*$, where $\Gamma_f \subset Y \times \bA^1$ is the usual graph of $f$.
\end{definition}

In equations, $\Gamma^a_f = V(z-f) \subset Y \times \bC^*_z$. Recall that the usual graph $\Gamma_f$ of $f$ is
isomorphic to $Y$ via the projection $Y \times \bA^1 \to Y$; likewise $\Gamma^a_f$ is isomorphic to $Y \setminus V(f)$
via the projection $Y \times \bC^* \to Y$.

\begin{theorem} \label{thm:qlin_vars}
    Quasilinear varieties are smooth, irreducible, rational, Chow-free, and linearly stratified.
\end{theorem}

\begin{proof}
    Let $\wY \subset \wT$ be a quasilinear variety. Following \cref{rmk:trop_iso}, we reduce to the case where $\wT$ is
    the minimal torus of $\wY$. If $\trop(\wY) = \wN_{\bR}$, then $\wY=\wT$ and we are done. Otherwise, $\trop(\wY)$ is
    a tropical modification of a quasilinear tropical fan cycle $\cF$ in $N_{\bR}$, $\rk N = \rk \wN - 1$, along a
    quasilinear tropical divisor $\cD$. By \cref{lem:mod_bir}, the corresponding projection of tori restricts to a
    birational morphism with finite fibers $\pi : \wY \to Y$ to some variety $Y \subset T$ with $\trop(Y) = \cF$ (as
    tropical fan cycles). Then $Y$ is quasilinear, so by induction on the dimension of the ambient torus, $Y$ is smooth,
    irreducible, rational, Chow-free, and linearly stratified. In particular, $Y$ is normal, so by a version of
    Zariski's main theorem, $\pi : \wY \to Y$ is an open embedding \cite[Ch. 4, Corollary
    4.6]{liuAlgebraicGeometryArithmetic2002}. Since $\pi$ is a morphism of affine varieties, it is also affine, so
    $\pi(\wY) \cong \wY$ is a dense affine open subvariety of $Y$. Therefore, by
    \cite[\href{https://stacks.math.columbia.edu/tag/0BCV}{Lemma 0BCV}]{stacks-project}, $D = Y \setminus \pi(\wY)$ is
    either empty or has codimension one in $Y$. If $D$ is empty, then $\pi : \wY \to Y$ is an isomorphism, and
    $\trop(\wY) \to \trop(Y)$ is a degenerate modification, so we're done.  Suppose $D$ is nonempty, and give $D$ the
    reduced induced scheme structure. Since $Y$ is Chow-free, every divisor on $Y$ is principal, so $D = V(f)$ for some
    regular function $f$ on $Y$. Then $\wY$ is the very affine graph of $h=fg$, where $g$ is some regular function on
    $Y$ such that either $g$ is nonvanishing or the support of $\div(g)$ is contained in $D$. Note that $V(h)$ is
    supported on $D$ but has a possibly nonreduced structure. By \cite[Proposition
    4.1]{renTropicalizationPezzoSurfaces2016}, $\trop(\wY)$ is a tropical modification of $\trop(Y)$ along
    $\trop(V(h))$. On the other hand, $\trop(\wY)$ is also a tropical modification of $\trop(Y)$ along $\cD$. Since the
    center of the modification is the image of the locus where $\trop(\wY) \not\cong \trop(Y)$, with the weights
    inherited from the $\trop(\wY)$, it follows that $\trop(V(h)) = \cD$ as tropical fan cycles. Then since $D$ is the
    reduced scheme structure on $V(h)$, it further follows that $\trop(D) = \cD$ as tropical fan cycles, hence $D$ is a
    quasilinear variety. We conclude that $\wY \cong Y \setminus D$ is the complement of a quasilinear divisor in a
    quasilinear variety. Since by induction both $Y$ and $D$ are smooth, irreducible, rational, Chow-free, and linearly
    stratified, it follows that $\wY$ is as well.
\end{proof}

We extract from the above theorem and its proof the following characterization of quasilinear varieties.

\begin{corollary}
    Let $\wY \subset \wT$ be a quasilinear variety. Then, up to tropical isomorphism, either $\wY$ is an algebraic
    torus, or there is a projection of tori $\pi : \wT \to T$, where $\dim T = \dim \wT - 1$, realizing $\wY$ as the
    very affine graph of a regular function $f$ on a quasilinear variety $Y \subset T$, such that $D = V(f)$ is also
    quasilinear.
\end{corollary}


\begin{example}
    \begin{enumerate}
        \item Let $\wY$ be the very affine graph of $f=z-1$ on $\bC^*$. Then $\trop(\wY)$ is the tropical modification
            of $\trop(\bC^*)=\bR$ with respect to the function $\trop(f)=\min\{x,0\}$. This yields the standard tropical
            line, hence $\wY$ is quasilinear, cf. \cref{fig:lines_R2}(b).
        \item Let $\wY'$ be the very affine graph of $g=z(z-1)$ on $\bC^*$. Then $\trop(\wY')$ is the tropical
            modification of $\trop(\bC^*)=\bR$ with respect to the function $\trop(g)=\min\{2x,x\}$. This yields the
            one-dimensional reduced tropical fan $\Sigma$ in $\bR^2$ with rays $(1,1)$, $(0,1)$, and $(-1,-2)$. The linear map
            $\bR^2 \to \bR^2$ given by the matrix
            \[
                \begin{bmatrix}
                    1 & 0 \\
                    1 & 1
                \end{bmatrix}
            \]
            induces an isomorphism of the standard tropical line with $\Sigma$, hence $\wY'$ is quasilinear.
        \item Let $\wY''$ be the very affine graph of $h=(z-1)^2$ on $\bC^*$. Then $\trop(\wY'')$ is the tropical
            modification of $\trop(\bC^*)=\bR$ with respect to the function $\trop(h)=\min\{2x,x,0\}$. This yields the
            nonreduced tropical fan shown in \cref{fig:nonreduced_R2}. Thus $\wY''$ is not quasilinear.
    \end{enumerate}
    In particular, observe that $\wY \cong \wY' \cong \wY'' \cong \bC^* \setminus \{1\}$, but quasilinearity depends on
    the given embedding of $\bC^* \setminus \{1\}$ in $(\bC^*)^2$: $\wY$ and $\wY'$ are quasilinear, while $\wY''$ is
    not.
\end{example}

\section{Quasilinear tropical compactifications} \label{sec:qlin_comps}

\subsection{Strata}

\begin{definition}
    A tropical compactification $\oY \subset X(\Sigma)$ of a closed subvariety $Y \subset T$ is \emph{quasilinear} if $Y
    \subset T$ is a quasilinear variety.
\end{definition}

\begin{theorem} \label{thm:qlin_strata}
    Suppose $\oY \subset X(\Sigma)$ is a quasilinear tropical compactification. Then for every cone $\sigma \in \Sigma$,
    the stratum $Y_{\sigma} = \oY \cap O(\sigma)$ of $\oY$ is quasilinear, and in particular smooth, irreducible,
    rational, Chow-free, and linearly stratified.
\end{theorem}

\begin{proof}
    If $\sigma=0$ then this follows by definition. Suppose $\sigma \neq 0$. By \cref{lem:trop_strata}, $Y_{\sigma}$ is
    pure-dimensional, and $\trop(Y_{\sigma})$ is a tropical fan cycle supported on $\Sigma^{\sigma}$, with induced
    weight from $Y_{\sigma}$ the same as the induced weight from $\Sigma$. Since star fans of quasilinear tropical fans
    are quasilinear (\cref{thm:star_qlin}), and in particular reduced and irreducible, it follows by
    \cref{lem:trop_integral} that $Y_{\sigma}$ is irreducible and generically reduced. Thus if $Y_{\sigma}^{red}$ is the
    reduced induced scheme structure on $Y_{\sigma}$, then $Y_{\sigma}^{red}$ irreducible and reduced. Furthermore,
    $\trop(Y_{\sigma}^{red}) = \trop(Y_{\sigma})$ as tropical fan cycles, and in particular $Y_{\sigma}^{red}$ is a
    quasilinear variety, hence is smooth, irreducible, rational, Chow-free, and linearly stratified by
    \cref{thm:qlin_vars}. Now it follows by \cref{rmk:in_degens} and \cite[Theorem
    11]{cartwrightGrobnerStratificationTropical2012} that $Y_{\sigma} = Y_{\sigma}^{red}$.
\end{proof}

\begin{corollary} \label{cor:qlin_schon}
    Quasilinear varieties are sch\"on.
\end{corollary}

\begin{proof}
    By the theorem, if $\oY \subset X(\Sigma)$ is any quasilinear tropical compactification, then all strata
    $Y_{\sigma}$ of $\oY$ are quasilinear, and in particular smooth. The result follows by \cite[Lemma
    2.7]{hackingHomologyTropicalVarieties2008}.
\end{proof}

\subsection{Chow and cohomology rings} \label{sec:chow}

We now explain our interest in Chow-free and linearly stratified varieties, following \cite[Section
5.2]{fultonIntroductionToricVarieties1993} and \cite{fultonIntersectionTheorySpherical1995, totaroChowGroupsChow2014}.

The following lemma is well-known.

\begin{lemma}
    Let $Z$ be a variety with a stratification by locally closed strata which are all Chow-free. Then $A_k(Z)$ is
    generated by the classes of the closures of the $k$-dimensional strata.
\end{lemma}

\begin{proof}
    We induct on $n = \dim Z$. When $n=0$ the claim is obvious. Suppose the result holds for all $i < n$.

    Let $Z_i$ be the union of the closed strata of dimension $\leq i$. This gives a filtration
    \[
      \emptyset = Z_{-1} \subset Z_0 \subset Z_1 \subset \cdots \subset Z_n = Z.
    \]
    Each $Z_i \setminus Z_{i-1}$ is the disjoint union of the strata of dimension $i$.

    Consider the exact sequence
    \[
      A_k(Z_{i-1}) \to A_k(Z_i) \to A_k(Z_i \setminus Z_{i-1}) \to 0.
    \]

    If $k=i$, then for dimension reasons $A_i(Z_{i-1}) = 0$, so $A_i(Z_i) \cong A_i(Z_i \setminus Z_{i-1})$, hence by
    assumption is generated by the classes of the closures of the $i$-dimensional strata.

    If $k < i$, then by assumption $A_i(Z_{i} \setminus Z_{i-1}) = 0$, so $A_k(Z_i)$ is generated by $A_k(Z_{i-1})$.

    By induction, $A_k(Z)$ is generated by $A_k(Z_k)$, which is generated by the classes of the closures of the
    $k$-dimensional strata.
\end{proof}

\begin{corollary} \label{cor:chow_groups}
    \begin{enumerate}
        \item For a toric variety $X(\Sigma)$, $A_{n-d+k}(X(\Sigma))$ is generated by the classes of the torus orbit closures
            $V(\sigma)$, for $\sigma \in \Sigma_{d-k}$.
        \item For a tropical compactification $\oY \subset X(\Sigma)$, if all strata $Y_{\sigma}$ are irreducible and
            Chow-free, then $A_k(\oY)$ is generated by the classes of the closures $\oY_{\sigma} = \oY \cap V(\sigma)$
            for $\sigma \in \Sigma_{d-k}$.
    \end{enumerate}
\end{corollary}

\begin{proof}
    Immediate from the lemma.
\end{proof}

\begin{lemma}[{\cite{totaroChowGroupsChow2014, fultonIntersectionTheorySpherical1995}}]
    \label{lem:weakly_linear}
    Let $\oY$ be a linearly stratified variety.
    \begin{enumerate}
        \item (Chow-K\"unneth) If $Z$ is any finite-type scheme, then $A_*(\oY) \otimes A_*(Z) \to A_*(\oY \times Z)$ is
            an isomorphism.
        \item Assume $\oY$ is proper.
            \begin{enumerate}
                \item (Kronecker-Poincar\'e duality) The natural map $A^k(\oY) \to \Hom(A_k(\oY),\bZ)$ is an isomorphism
                    for all $k$. In particular, if $\oY$ is smooth, then $A^*(\oY)$ is a Poincar\'e duality ring of
                    dimension $\dim \oY$.
                \item (Homology isomorphism) If $\oY$ is nonsingular, then the cycle class map $A_k(\oY) \cong
                    H_{2k}(\oY)$ is an isomorphism and $H_{2k+1}(\oY) =0$ for all $k$, inducing an isomorphism $A_*(\oY)
                    \cong H_*(\oY)$.
            \end{enumerate}
    \end{enumerate}
\end{lemma}

\begin{proof}
    The first two parts are \cite[Propositions 1,2]{totaroChowGroupsChow2014}.  (The second part in fact follows from
    the first part. Totaro's proof of the first part is actually for a slighter narrower class of varieties, but his
    proof also works for linearly stratified varieties as we have defined them, cf. \cite[Comments after Definition
    2.3]{gonzalesEquivariantOperationalChow2015}.) The last part is also a consequence of the first part \cite[Corollary
    to Theorem 2]{fultonIntersectionTheorySpherical1995}.
\end{proof}

For the remainder of this section fix any $d$-dimensional closed subvariety $Y \subset T \cong (\bC^*)^n$ and a tropical
compactification $i : \oY \hookrightarrow X(\Sigma)$ of $Y$, and let $\omega$ be the weight on $\Sigma$ induced by $Y$.
Recall that $\omega(\sigma)$ can be defined as $\int_{\oY} i^*[V(\sigma)]$ \cite[Theorem
6.7.5]{maclaganIntroductionTropicalGeometry2015}, and via the isomorphism $M_d(\Sigma) \cong \Hom(A_{n-d}(X(\Sigma)),\bZ)$,
we view $\omega$ as a homomorphism $A_{n-d}(X(\Sigma)) \to \bZ$.

Since $i : \oY \hookrightarrow X(\Sigma)$ is a tropical compactification, the multiplication map $m : \oY \times T \to
X(\Sigma)$ is flat, of relative dimension $d$. Therefore, there is a flat pullback morphism \cite[Section
1.7]{fultonIntersectionTheory1998}
\begin{align*}
    m^* : A_{n-d+k}(X(\Sigma)) &\to A_{k+n}(\oY \times T), \\
    [V(\sigma)] &\mapsto [m^{-1}(V(\sigma))] = [(\oY \cap V(\sigma)) \times T].
\end{align*}
There is a natural isomorphism $A_*(\oY \times T) \cong A_*(\oY) \otimes A_*(T) \cong A_*(\oY)$, under which the above
map is identified with
\begin{align*}
    i^* : A_{n-d+k}(X(\Sigma)) &\to A_{k}(\oY), \\
    [V(\sigma)] &\mapsto [\oY \cap V(\sigma)].
\end{align*}

\begin{corollary} \label{cor:chow_sur}
    If all strata $Y_{\sigma}$ of $\oY \subset X(\Sigma)$ are irreducible and Chow-free, then $i^* :
    A_{n-d+k}(X(\Sigma)) \to A_k(\oY)$ is surjective for all $k$.
\end{corollary}

\begin{proof}
    Immediate from the above discussion and \cref{cor:chow_groups}.
\end{proof}

Taking duals, we get a morphism
\[
    i_* : \Hom(A_k(\oY),\bZ) \to \Hom(A_{n-d+k}(X(\Sigma)),\bZ) \cong M_{d-k}(\Sigma).
\]

\begin{proposition}
    The following diagram commutes
    \begin{equation} \label{eq:key_diag}
        \begin{tikzcd}
            A^k(\Sigma) \ar[r, "- \cap \omega"] \ar[d, "i^*"] & M_{d-k}(\Sigma) \\
            A^k(\oY) \ar[r, "\cD_{\oY}"] & \Hom(A_k(\oY),\bZ) \ar[u, "i_*"],
        \end{tikzcd}
    \end{equation}
    where the maps are defined as follows.
    \begin{enumerate}
        \item $- \cap \omega : A^k(\Sigma) \to M_{d-k}(\Sigma)$ is the tropical cap product.
        \item $i^* : A^k(\Sigma) = A^k(X(\Sigma)) \to A^k(\oY)$ is the natural pullback of Chow cohomology groups.
        \item $\cD_{\oY} : A^k(\oY) \to \Hom(A_k(\oY),\bZ)$ is the ``Kronecker-Poincar\'e'' duality map, $\alpha \mapsto
            \left(\beta \mapsto \int_{\oY} \alpha \cap \beta\right)$.
        \item $i_* : \Hom(A_k(\oY),\bZ) \to \Hom(A_{n-d+k}(X(\Sigma)),\bZ) \cong M_{d-k}(\Sigma)$ is the dual to the
            pullback morphism $A_{n-d+k}(X(\Sigma)) \to A_{k+n}(\oY \times T) \xrightarrow{\sim} A_k(\oY)$ defined
            above.
    \end{enumerate}
\end{proposition}

\begin{proof}
    Unwinding definitions, we are simply asking that for any $\alpha \in A^k(X(\Sigma))$, $\beta \in
    A_{n-d+k}(X(\Sigma))$, one has
    \[
        \omega(\alpha \cap \beta) = \int_{\oY} i^*\alpha \cap i^*\beta,
    \]
    where $i^*\alpha$ is the pullback on Chow cohomology groups and $i^*\beta$ is the pullback on Chow homology
    groups as defined above. But $i^*\alpha \cap i^*\beta = i^*(\alpha \cap \beta)$, where now $i^*(\alpha \cap \beta)$
    is the pullback $i^*: A_{n-d}(X(\Sigma)) \to A_0(\oY)$. The desired equality now follows from the definition of
    $\omega$.
\end{proof}

\begin{corollary} \label{cor:chow_inj}
    If $(\Sigma,\omega)$ is Poincar\'e, then $i^* : A^k(\Sigma) \to A^k(\oY)$ is injective and $i_* : \Hom(A_k(\oY),\bZ)
    \to M_{d-k}(\Sigma)$ is surjective for all $k$.
\end{corollary}

\begin{proof}
    Immediate from the proposition.
\end{proof}

\begin{corollary} \label{cor:chow_iso}
    If $(\Sigma,\omega)$ is Poincar\'e and all strata $Y_{\sigma}$ of $\oY$ are irreducible and Chow-free, then $i_* :
    \Hom(A_k(\oY),\bZ) \to M_{d-k}(\Sigma)$ is an isomorphism for all $k$. If in addition, either all strata
    $Y_{\sigma}$ are linearly stratified, or both $\oY$ and $X(\Sigma)$ are nonsingular, then $i^* : A^k(\Sigma) \to A^k(\oY)$
    and $\cD_{\oY} : A^k(\oY) \to \Hom(A_k(\oY),\bZ)$ are both isomorphisms, i.e., all arrows in \eqref{eq:key_diag} are
    isomorphisms.
\end{corollary}

\begin{proof}
    The first sentence follows from \cref{cor:chow_sur,cor:chow_inj}. For the second, first suppose all strata
    $Y_{\sigma}$ are linearly stratified. Then by \cref{lem:weakly_linear}, the Kronecker-Poincar\'e duality map $\cD_{\oY}:
    A^k(\oY) \to \Hom(A_k(\oY),\bZ)$ is an isomorphism, i.e., all arrows in \eqref{eq:key_diag} are isomorphisms, except
    possibly $i^* : A^k(\Sigma) \to A^k(\oY)$. But of course, it now follows that this is also an isomorphism. If
    instead both $\oY$ and $X(\Sigma)$ are nonsingular, then $A_{n-d+k}(X(\Sigma)) \cong A^{d-k}(X(\Sigma))$ and
    $A_k(\oY) \cong A^{d-k}(\oY)$, so the map $i_* : \Hom(A_k(\oY),\bZ) \to M_{d-k}(\Sigma) \cong
    \Hom(A_{n-d+k}(X(\Sigma)),\bZ)$ is just the dual to $i^* : A^{d-k}(\Sigma) \to A^{d-k}(\oY)$, and the result
    follows.
\end{proof}

\begin{corollary} \label{cor:hom_iso}
    Suppose all the conditions of \cref{cor:chow_iso} hold, i.e., $\Sigma$ is Poincar\'e, all strata $Y_{\sigma}$ of
    $\oY$ are irreducible, Chow-free, and linearly stratified, and both $\oY$ and $X(\Sigma)$ are nonsingular. Then the cycle
    class map $A_k(\oY) \to H_{2k}(\oY)$ is an isomorphism and $H_{2k+1}(\oY)=0$ for all $k$, inducing isomorphisms
    \begin{align*}
        H^*(\oY) \cong A^*(\oY) \cong A^*(X(\Sigma)).
    \end{align*}
\end{corollary}

\begin{proof}
    Immediate from \cref{lem:weakly_linear,cor:chow_iso}.
\end{proof}

\begin{theorem} \label{thm:qlin_chow}
    If $i : \oY \hookrightarrow X(\Sigma)$ is a quasilinear tropical compactification, then $i^* : A^k(\Sigma) \to
    A^k(\oY)$ is an isomorphism for all $k$, inducing an isomorphism of Chow rings
    \[
        A^*(\oY) \cong A^*(X(\Sigma)).
    \]
    If in addition $X(\Sigma)$ is nonsingular, then so is $\oY$, and we have $H^{2k+1}(\oY)=0$ and
    $H^{2k}(\oY) \cong A^k(\oY) \cong A^k(X(\Sigma))$ for all $k$, inducing isomorphisms
    \begin{align*}
        H^*(\oY) \cong A^*(\oY) \cong A^*(X(\Sigma)).
    \end{align*}
\end{theorem}

\begin{proof}
    By \cref{thm:qlin_strata}, all strata $Y_{\sigma}$ of $\oY$ are smooth, irreducible, Chow-free, and linearly
    stratified, and by \cref{thm:qlin_poincare}, $\Sigma$ is (star-)Poincar\'e. The first statement is now immediate from
    \cref{cor:chow_iso}. For the second statement, if $X(\Sigma)$ is nonsingular then since $\oY \subset X(\Sigma)$ is
    a sch\"on tropical compactification, $\oY$ is also nonsingular, and the result follows by \cref{cor:hom_iso}.
\end{proof}

\section{Quasilinear criteria and examples} \label{sec:qlin_criteria}

\subsection{General criteria}

\begin{theorem} \label{thm:qlin_criterion_main}
    Let $Y \subset T$ be a quasilinear variety, and $f$ a regular function on $Y$ such that either $f$ is nonvanishing
    or $D = V(f) \subset Y$ is also quasilinear in $T$. Let $\wY \subset T \times \bC^*$ be the very affine graph of $f$
    on $Y$. Then $\trop(\wY)$ is a tropical modification of $\trop(Y)$ along $\trop(D)$, and in particular $\wY$ is
    quasilinear.
\end{theorem}

\begin{proof}
    We prove the result in the case that $D = V(f)$ is nontrivial, leaving the (easier) case that $f$ is nonvanishing to
    the reader.

    Choose unimodular fan structures $\Sigma$ on $\trop(Y)$ and $\Delta$ on $\trop(D)$ such that every cone of $\Delta$
    is a cone of $\Sigma$. Since by \cref{cor:qlin_schon} quasilinear varieties are sch\"on, it follows that the
    closures $\oY \subset X(\Sigma)$ and $\oD \subset X(\Delta)$ are sch\"on tropical compactifications \cite[Theorem
    1.5]{luxtonResultsTropicalCompactifications2011}, and in particular (since $X(\Sigma)$ and $X(\Delta)$ are smooth),
    $Y \subset \oY$ and $D \subset \oD$ are smooth compactifications with simple normal crossings boundary \cite[Theorem
    1.4]{tevelevCompactificationsSubvarietiesTori2007}. By \cref{thm:qlin_strata}, all boundary strata of $\oY$ and
    $\oD$ are quasilinear, and in particular reduced and irreducible. Therefore the fans $\Sigma$ and $\Delta$ can be
    recovered from $\oY$ and $\oD$ via geometric tropicalization \cite[Section 2]{hackingStablePairTropical2009},
    \cite[Theorem 2.8]{cuetoImplicitizationSurfacesGeometric2012a}. Namely, a ray $\rho$ of $\Sigma$ is recovered from
    the corresponding irreducible boundary divisor $D_{\rho} = \oY_{\rho}$ as the ray through the vector
    $[\val_{D_{\rho}}] = (\val_{D_{\rho}}(m_1\vert_Y),\ldots,\val_{D_{\rho}}(m_n\vert_Y))$, where $m_1,\ldots,m_{n}$ are
    a basis of $M$, and a collection of such rays forms a cone precisely when the corresponding boundary divisors
    intersect; likewise for recovering $\Delta$ from $\oD$.  Furthermore, $\oD$ is also obtained as the closure of $D$
    in $\oY$, and $\oD$ intersects a stratum $Y_{\sigma}$ of $\oY \iff$ $\sigma \in \Delta$. It follows that $\oY$ is
    also a smooth simple normal crossings compactification of $\wY \cong Y \setminus D$, so we can also recover
    $\trop(\wY)$ from the boundary of $\wY \subset \oY$ via geometric tropicalization. Now the irreducible components of
    the boundary are $\oD$, and $D_{\rho} = \oY_{\rho}$ for $\rho$ a ray of $\Sigma$, a collection of boundary divisors
    $D_{\rho_1},\ldots,D_{\rho_k}$ intersect $\iff \rho_1,\ldots,\rho_k$ form a cone of $\Sigma$, and a collection
    of boundary divisors $\oD, D_{\rho_1},\ldots,D_{\rho_k}$ intersect $\iff \rho_1,\ldots,\rho_k$ form a cone of
    $\Delta$. So geometric tropicalization implies that $\trop(\wY)$ is the support of the fan $\wt\Sigma \subset
    \wN_{\bR} =N_{\bR} \times \bR$, whose rays are
    \begin{align*}
        \wt\rho &=
        \bR_{\geq 0}(\val_{D_{\rho}}(m_1\vert_{\wY}),\ldots,\val_{D_{\rho}}(m_n\vert_{\wY}),\val_{D_{\rho}}(m_{n+1}\vert_{\wY})) \\
                &= \{(x,\varphi(x)) \mid x \in \rho\}, \\
        0_{\geq} &= \bR_{\geq
        0}(\val_{\oD}(m_1\vert_{\wY}),\ldots,\val_{\oD}(m_n\vert_{\wY}),\val_{\oD}(m_{n+1}\vert_{\wY})) \\
                 &= \bR_{\geq 0}(0,\ldots,0,1),
    \end{align*}
    where $\varphi : \lvert \Sigma \rvert \to \bR$ is the piecewise integral linear function defined by setting
    $\varphi(v_{\rho}) = \val_{D_{\rho}}(m_{n+1}\vert_{\wY})$ for each ray $\rho$, and extending by linearity on each
    cone of $\Sigma$. From the above description of the boundary complex of $\wY \subset \oY$, we see that a collection
    of rays $\wt\rho_1,\ldots,\wt\rho_k$ form a cone of $\wt\Sigma$ precisely if the corresponding rays
    $\rho_1,\ldots,\rho_k$ form a cone of $\Sigma$, and a collection of rays $0_{\geq}, \wt\rho_1,\ldots,\wt\rho_k$
    form a cone of $\wt\Sigma$ precisely if the corresponding rays $\rho_1,\ldots,\rho_k$ form a cone of $\Delta$. It
    follows that $\wt\Sigma$ is precisely the tropical modification of $\Sigma$ with respect to the piecewise integral
    linear function $\varphi$, and $\div(\varphi)=\Delta$.
\end{proof}

\begin{remark}
    A more general variation of the above theorem appears in \cite[Proposition
    4.1]{renTropicalizationPezzoSurfaces2016}.
\end{remark}

\begin{corollary}
    Let $Y \subset T$ be a quasilinear variety and $D_1,\ldots,D_k \subset Y$ quasilinear hypersurfaces on $Y$ such that
    all nonempty intersections of the $D_i$ are also quasilinear in $T$. Then $Y \setminus (D_1 \cup \cdots \cup D_k)
    \subset T \times T^k$ is also quasilinear.
\end{corollary}

\begin{proof}
    The proof is by induction on $k$. The base case $k=1$ is immediate from \cref{thm:qlin_criterion_main}. Assume the
    result for $k-1$.  Then $Y' = Y \setminus (D_1 \cup \cdots \cup D_{k-1})$ is quasilinear by induction, and $Y
    \setminus (D_1 \cup \cdots \cup D_k) = Y' \setminus D_k'$, wher $D_k' = D_k \setminus ((D_k \cap D_1) \cup \cdots
    \cup (D_k \cap D_{k-1}))$. Since $D_k$ and all intersections of the $D_i$'s are quasilinear, $D_k'$ is also
    quasilinear by induction. Thus $Y' \setminus D_k'$ is quasilinear.
\end{proof}

\begin{theorem}
    Let $Y_1 \subset T_1$ and $Y_2 \subset T_2$ be two closed subvarieties of tori.  Then $Y_1 \times Y_2 \subset T_1
    \times T_2$ is quasilinear, if and only if both $Y_1 \subset T_1$ and $Y_2 \subset T_2$ are quasilinear.
\end{theorem}

\begin{proof}
    We have $\trop(Y_1 \times Y_2) = \trop(Y_1) \times \trop(Y_2)$ \cite[Theorem
    3.3.4]{cuetoTropicalImplicitization2010}, so the result is immediate from \cref{thm:prod_qlin}.
\end{proof}

%
%

\subsection{Examples} \label{sec:examples_vars}

Recall a closed subvariety $Y \subset T$ is \emph{linear} if the equations defining $Y$ are linear, or equivalently $Y$
is isomorphic to the complement of a hyperplane arrangement.

\begin{theorem} \label{thm:lin_qlin_var}
    Linear varieties are quasilinear.
\end{theorem}

\begin{proof}
    The tropicalization of a linear variety is the (support of the) Bergman fan of the matroid defining the
    corresponding hyperplane arrangement, so the result follows by \cref{thm:lin_qlin}.
\end{proof}

\begin{example}
    Let $a,b \in \bZ \setminus \{0\}$, $m \in \bC^*$, and let $Y = \{x^a=my^b\} \subset (\bC^*)^2$. Then $\trop(Y)$ is
    the line $\{ax=by\} \subset \bR^2$, hence $Y$ is quasilinear by \cref{ex:arb_line}.

    Similarly, if $c \in \bZ \setminus \{0\}$, then $Y = \{z^c=mx^ay^b\} \subset (\bC^*)^3$ is quasilinear; its
    tropicalization is a classical plane in $\bR^3$.
\end{example}

\begin{example}
    Let $a,b \in \bZ \setminus \{0\}$, $m,n \in \bC^*$. Then $Y = \{z=mx^a-ny^b\} \subset (\bC^*)^3$ is quasilinear; its
    tropicalization is a modification of $\bR^2$ along the line $\{ax=by\}$.
\end{example}

\begin{example}
    Let $H_1 = \{x=y\}$, $H_2 = \{xy=1\}$, and let $Y = (\bC^*)^2 \setminus (H_1 \cup H_2)$. Then
    \[
        Y = \{z=x-y, w=xy-1\} \subset (\bC^*)^4
    \]
    is not quasilinear---its tropicalization has a top-dimensional cone of with weight two, corresponding to the two
    intersection points of $H_1$ and $H_2$.
\end{example}

\begin{example}
    Continuing with the previous example, let $H_1 = \{x=y\}$, $H_2 = \{xy=1\}$, $H_3 = \{x=1\}$, and let $\wY = (\bC^*)^2
    \setminus (H_1 \cup H_2 \cup H_3)$. Then $\wY$ is quasilinear, even though $Y = (\bC^*)^2 \setminus (H_1 \cup H_2)$
    is not. To see this, let $Y' = (\bC^*)^2 \setminus (H_1 \cup H_3)$. Then
    \[
        Y' = \{z=x-y, w=x-1\} \subset (\bC^*)^4
    \]
    is linear, hence quasilinear. Let
    \[
        H_2' = H_2 \cap Y' = \{xy=1, z=x-y, w=x-1\} \subset (\bC^*)^4.
    \]
    Then $\wY = Y' \setminus H_2'$, so we just need to show that $H_2'$ is quasilinear. But the projection $(\bC^*)^4
    \to (\bC^*)^3$ dropping the $z$ coordinate realizes $H_2'$ as the graph of the regular function $f=x-y$ on the variety
    \[
        H_2'' = \{xy=1, w=x-1\} \subset (\bC^*)^3.
    \]
    The projection $(\bC^*)^3 \to (\bC^*)^2$ dropping the coordinate $y$ realizes $H_2''$ as the graph of the
    nonvanishing regular function $1/x$ on the (quasi)linear variety
    \[
        \{w=x-1\} \subset (\bC^*)^2,
    \]
    thus $H_2''$ is quasilinear. Now $V = V(f) \subset H_2''$ is given by
    \[
        \{x=y, xy=1, w=x-1\} = \{x=y=-1,w=-2\} \subset (\bC^*)^3,
    \]
    hence $V$ is (quasi)linear. (Indeed, it is a single point.) It follows that $H_2'$ is quasilinear, hence $\wY = Y'
    \setminus H_2'$ is quasilinear.
\end{example}

The above example is representative of our general procedure for determining when a variety is quasilinear, see
\cref{sec:applications}.

\section{Applications} \label{sec:applications}

\subsection{Moduli of line arrangements}

For $n \geq 4$, let $M(3,n)$ denote the moduli space of arrangements of $n$ lines in general position in $\bP^2$. By
fixing the last 4 lines and scaling the equations of the remaining lines, we can write $M(3,n)$ as the complement in
$\bA^{2(n-4)}$ of the $3 \times 3$ minors of the matrix
\[
    \begin{bmatrix}
        1 & 1 & \cdots & 1 & 1 & 1 & 0 & 0 \\
        x_1 & x_2 & \cdots & x_{n-4} & 1 & 0 & 1 & 0 \\
        y_1 & y_2 & \cdots & y_{n-4} & 1 & 0 & 0 & 1
    \end{bmatrix}.
\]
The $3 \times 3$ minors of this matrix include the hyperplanes $x_i=0$ and $y_i=0$, inducing the embedding of $M(3,n)$
in its intrinsic torus $(\bC^*)^{\binom{n}{3}-n}$.

\begin{theorem} \label{thm:M36_quasilinear}
    The moduli space $M(3,6)$ is quasilinear.
\end{theorem}

\begin{proof}
    Observe that $M=M(3,6)$ is the complement in $(\bC^*)^4$ of the 10 hypersurfaces defined by the following 10
    equations.
    \begin{align*}
        x_1=1, \;\; x_2=1, \;\; x_1=x_2, \;\; y_1=1, \;\; y_2=1, \;\; y_1=y_2, \\
        x_1=y_1, \;\; x_2=y_2, \;\; x_1y_2=x_2y_1, \;\; (x_1-1)(y_2-1)=(x_2-1)(y_1-1).
    \end{align*}
    Let $M_1 \subset (\bC^*)^{12}$ be the complement of the linear hypersurfaces. Then $M_1$ is linear, and in
    particular quasilinear. Equations for $M_1 \subset (\bC^*)^{12}$ can be written as follows.
    \begin{align*}
        z_1=x_1-1, \;\; z_2=x_2-1, \;\; z_{12}=x_1-x_2, \\
        w_1=y_1-1, \;\; w_2=y_2-1, \;\; w_{12}=y_1-y_2, \\
        u_1=x_1-y_1, \;\; u_2=x_2-y_2.
    \end{align*}
    Let $Q_1 = \{x_1y_2=x_2y_1\} \subset M_1$ and $Q_2 = \{(x_1-1)(y_2-1)=(x_2-1)(y_1-1)\} = \{z_1w_2=z_2w_1\} \subset
    M_1$. Then $M = M_1 \setminus (Q_1 \cup Q_2)$, so to show $M$ is quasilinear it suffices to show that $Q_1$, $Q_2$,
    and $Q_1 \cap Q_2$ are all quasilinear. A direct computation shows that in fact $Q_1 \cap Q_2 = \emptyset$, so it is
    enough to show that $Q_1$ and $Q_2$ are quasilinear. But by the symmetry of the equations, we see that it is enough
    to check that $Q_1$ is quasilinear.

    Let $Q = \{x_1y_2=x_2y_1\} \subset (\bC^*)^4$. Then $Q_1$ is the complement in $Q$ of the linear hypersurfaces. So
    to show $Q_1 \subset (\bC^*)^{14}$ is quasilinear, it suffices to show $Q \subset (\bC^*)^4$ intersects any
    collection of the linear hypersurfaces quasilinearly. Up to symmetry, there are three types of intersections of $Q$
    with a single linear hypersurface:
    \begin{enumerate}
        \item $Q \cap \{x_1=y_1\} = \{x_1=y_1,x_2=y_2\}$ is linear.
        \item $Q \cap \{x_1=x_2\} = \{x_1=x_2,y_1=y_2\}$ is linear.
        \item $Q \cap \{x_1=1\} = \{x_1=1, y_2=x_2y_1\}$. The projection dropping the coordinate $y_2$ realizes this as
            the graph of the nonvanishing regular function $x_2y_1$ on the (quasi)linear variety $\{x_1=1\} \subset
            (\bC^*)^3$, hence $Q \cap \{x_1=1\}$ is quasilinear.
    \end{enumerate}
    Finally, any intersection of $Q$ with two or more linear hypersurfaces is linear.  Thus, $Q_1$ is linear.
\end{proof}

As an immediate consequence of \cref{thm:M36_quasilinear}, we recover the main result of
\cite{luxtonLogCanonicalCompactification2008}, describing the log canonical compactification of $M(3,6)$.

\begin{corollary} \label{cor:M36_schon}
    The moduli space $M(3,6)$ is sch\"on, and its stable pair compactification $\oM(3,6)$ is the log canonical
    compactification.
\end{corollary}

\begin{proof}
    Sch\"onness of $M(3,6)$ is immediate from \cref{thm:M36_quasilinear} and \cref{cor:qlin_schon}. The stable pair compactification $\oM(3,6)$ is
    obtained as the closure of $M(3,6)$ in the toric variety associated to the Dressian $Dr(3,6)$ (see
    \cite{hackingCompactificationModuliSpace2006, alexeevModuliWeightedHyperplane2015} for more details). The Dressian
    $Dr(3,6)$ is a convexly disjoint fan supported on $\trop(M(3,6))$ (meaning that any convex subset of $Dr(3,6)$ is
    contained in a cone of $Dr(3,6)$). This together with the sch\"onness of $M(3,6)$ implies that $\oM(3,6)$ is the log
    canonical compactification, see \cite[Proof of Theorem 9.14]{hackingStablePairTropical2009},
    \cite{luxtonLogCanonicalCompactification2008}.
\end{proof}

\begin{remark}
    In general, let $M(r,n)$ denote the moduli space of arrangements of $n$ hyperplanes in $\bP^{r-1}$ in general
    position, and $\oM^m(r,n)$ the normalization of the main irreducible component of its stable pair compactification
    $\oM(r,n)$ \cite{hackingCompactificationModuliSpace2006}. In \cite{keelGeometryChowQuotients2006}, Keel and Tevelev
    show that $\oM^m(r,n)$ is not log canonical except possibly in the cases $(r,n)=(2,n)$, $(3,6)$, $(3,7)$, $(3,8)$
    (and those obtained by duality $\oM^m(r,n) \cong \oM^m(n-r,n)$), and they conjecture that in these cases
    $\oM^m(r,n)$ is indeed the log canonical compactification. They prove this for the case $r=2$, in which case
    $\oM^m(2,n) \cong \oM_{0,n}$ \cite{keelGeometryChowQuotients2006}. The case $(3,6)$ was shown by Luxton
    \cite{luxtonLogCanonicalCompactification2008}; the above corollary gives a much shorter proof of Luxton's result.
    (For another proof of this case, without using any tropical geometry, see \cite[Section
    7]{schockIntersectionTheoryStable2022}.) The cases $(3,7)$ and $(3,8)$ have recently been proven by Corey
    \cite{coreyInitialDegenerationsGrassmannians2021}, and Corey-Luber \cite{coreyGrassmannianPlanesMathbb2023}.
\end{remark}

\begin{remark}
    Corey and Luber show that for $(r,n)=(3,8)$ (and hence for any larger case), $\trop(M(3,8))$ is not reduced, and in
    particular $M(3,8)$ is not quasilinear \cite{coreyGrassmannianPlanesMathbb2023}. The nonreducedness of
    $\trop(M(3,8))$ comes from zero-dimensional strata of $\oM(3,8)$ parameterizing (stable replacements of)
    configurations of 8 lines in $\bP^2$ with 8 triple intersection points. Since the dual of such a configuration
    admits the same description, one sees that these strata consist of 2 distinct points, and thus are not irreducible.
    See \cite{coreyGrassmannianPlanesMathbb2023} for more details.

    On the other hand, such examples do not occur for $(r,n)=(3,7)$, and one can show by similar arguments to the proof
    of \cref{thm:M36_quasilinear} that $M(3,7)$ is quasilinear. The arguments in this case are quite subtle and do not
    offer much insight into the structure of $M(3,7)$ so we omit them, referring to \cite[Chapter
    7]{schockGeometryTropicalCompactifications2022} for more details.
\end{remark}

As another immediate consequence of \cref{thm:M36_quasilinear}, we recover the main results of
\cite{schockIntersectionTheoryStable2022}, describing the Chow rings of tropical compactifications of $M(3,6)$.

\begin{corollary} \label{cor:M36_chow}
    Let $\oM^{\Sigma}(3,6) \subset X(\Sigma)$ be any tropical compactification of $M(3,6)$. Then $A^*(\oM^{\Sigma}(3,6))
    \cong A^*(X(\Sigma))$. If $\Sigma$ is unimodular, then $\oM^{\Sigma}(3,6)$ is a resolution of singularities of the
    stable pair compactification $\oM(3,6)$, and $A^*(\oM^{\Sigma}(3,6)) \cong H^*(\oM^{\Sigma}(3,6))$.
\end{corollary}

\begin{proof}
    Everything is immediate from \cref{thm:M36_quasilinear} and \cref{thm:qlin_chow} except for the statement that if
    $\Sigma$ is unimodular then $\oM^{\Sigma}(3,6)$ is a resolution of singularities of $\oM(3,6)$; this in turn follows
    from the observation that $\oM(3,6)$ is the tropical compactification associated to the coarsest fan structure on
    $\trop(M(3,6))$ (cf. the proof of \cref{cor:M36_schon}, and \cite{luxtonLogCanonicalCompactification2008}).
\end{proof}

\begin{remark}
    The moduli space $\oM(3,6)$ has 15 singular points, each locally isomorphic to the cone over $\bP^1 \times \bP^2$
    (see \cite{luxtonLogCanonicalCompactification2008, schockIntersectionTheoryStable2022}).  There are $2^{15}$ small
    resolutions of $\oM(3,6)$, whose fibers over each singular point are either $\bP^1$ or $\bP^2$. In
    \cite{schockIntersectionTheoryStable2022} we give an explicit presentation of the Chow ring of each small
    resolution, and as a consequence also obtain a presentation of the Chow ring of $\oM(3,6)$ itself. Since each small
    resolution of $\oM(3,6)$ is also a tropical compactification of $M(3,6)$, the above result gives an alternative
    proof of these presentations.
\end{remark}

\subsection{Moduli of marked cubic surfaces}

Let $Y(E_6)$ denote the moduli space of smooth marked cubic surfaces. A marking of a cubic surface $S$ is equivalent to a
realization of $S$ as the blowup of 6 points in general position in $\bP^2$, where here general position means no 2
points coincide, no 3 lie on a line, and no 6 lie on a conic. Projective duality therefore identifies $Y(E_6)$ as the
complement in $M(3,6)$ of the hypersurface $Q$ parameterizing the locus of 6 points on a conic. In the coordinates of
\cref{sec:applications}, equations for $Q$ are given by
\[
    Q = \left\{\det \begin{pmatrix}
            1 & 1 & 1 \\
            x_1 & y_1 & x_1y_1 \\
            x_2 & y_2 & x_2y_2
    \end{pmatrix} = 0\right\}
\]
(see, for instance \cite{yoshidaE_6EquivariantProjective2000}). Note that $Q$ is also naturally isomorphic to the
moduli space $M_{0,6}$ of 6 points on $\bP^1$.

\begin{theorem} \label{thm:Y_quasilinear}
    The moduli space $Y(E_6)$ is quasilinear.
\end{theorem}

\begin{proof}
    Since $Y(E_6)$ is the complement in $M(3,6)$ of the hypersurface $Q$, and $M(3,6)$ is quasilinear by
    \cref{thm:M36_quasilinear}, all we need to do is show that $Q$ is quasilinear (in the embedding $Q \subset M(3,6)
    \subset (\bC^*)^{14}$). Since $Q \cong M_{0,6}$ is quasilinear (indeed, linear) in its intrinsic torus, it suffices
    to show that the tropicalization of $Q$ in $(\bC^*)^{14}$ is isomorphic to the tropicalization of $M_{0,6}$. In
    fact, this follows from the results of Luxton \cite{luxtonLogCanonicalCompactification2008}, which show that there
    is a tropical compactification $\wM(3,6)$ of $M(3,6)$ such that the closure of $Q$ in $\wM(3,6)$ is isomorphic to
    $\oM_{0,6}$.
\end{proof}

\begin{remark}
    Recall that the Segre cubic $\cS$ is the unique cubic threefold in $\bP^4$ with 10 nodes. It is a birational model
    of $\oM_{0,6}$ obtained by contracting the 10 boundary divisors of the form $D_{ijk}$ to the nodes. The \emph{open}
    Segre cubic $\cS^{\circ}$ is the complement in $\cS$ of the images of the remaining 15 boundary divisors $D_{ij}$ of
    $\oM_{0,6}$; note that $\cS^{\circ} \cong M_{0,6}$. This induces an embedding of $M_{0,6}$ into a 14-dimensional
    torus, which, up to an appropriate change of coordinates, agrees with our embedding $Q \subset M(3,6) \subset
    (\bC^*)^{14}$ above (cf. \cite[Section 2]{renTropicalizationClassicalModuli2014}). By \cite[Theorem
    2.4]{renTropicalizationClassicalModuli2014}, the tropicalization of this embedding is isomorphic to the
    tropicalization of $M_{0,6}$ in its intrinsic torus, giving another proof that $Q \subset (\bC^*)^{14}$ is
    quasilinear.

    It is also interesting to note that the closure $\oQ$ of $Q$ in $\oM(3,6)$ is in fact isomorphic to the Segre cubic.
    The 10 nodes of the Segre cubic are cut out by the intersections of $\oQ$ with the codimension two boundary strata
    of $\oM(3,6)$ described by intersections $D_{ijk} \cap D_{lmn}$, where $ijk \subset [6]$ and $lmn \subset [6]$ form
    a partition of $[6] = \{1,\ldots,6\}$, and the divisors $D_{ijk} \cong \oM_{0,6}$ are described in \cite[Section
    2]{schockIntersectionTheoryStable2022} (see also \cite{luxtonLogCanonicalCompactification2008}). Note that $\oQ$ is
    not a tropical compactification of $Q$ because of its intersections with these strata. Blowing up $\oM(3,6)$ at the
    10 intersections $D_{ijk} \cap D_{lmn}$ amounts to blowing up $\oQ$ at the 10 nodes, giving the resolution of
    singularities $\oM_{0,6} \to \oQ$. (This description of $\oQ$ is implicit in
    \cite{luxtonLogCanonicalCompactification2008}.)
\end{remark}

\begin{remark}
    For another proof that $Q \subset M(3,6) \subset (\bC^*)^{14}$ (hence $Y(E_6)$) is quasilinear, see \cite[Chapter
    8]{schockGeometryTropicalCompactifications2022}. The proof there works explicitly with the equations of $Q$.
\end{remark}

A particularly nice, smooth projective compactification $\oY(E_6)$ of $Y(E_6)$, with simple normal crossings boundary,
was constructed by Naruki in \cite{narukiCrossRatioVariety1982}. Using this compactification, Hacking-Keel-Tevelev show
that $Y(E_6)$ is sch\"on and $\oY(E_6)$ is the log canonical compactification \cite{hackingStablePairTropical2009}. They
also show that a natural blowup $\wY(E_6)$ of $\oY(E_6)$ is the moduli space of stable marked cubic surfaces. (In fact,
$\oY(E_6)$ also has an intepretation as a moduli space of \emph{weighted} stable marked cubic surfaces, see
\cite{gallardoGeometricInterpretationToroidal2021, schockModuliWeightedStable2024}.) \cref{thm:Y_quasilinear} gives an
alternative proof that $Y(E_6)$ is sch\"on, and allows one to describe the Chow rings of $\oY(E_6)$ and $\wY(E_6)$.

\begin{corollary} \label{cor:Y_schon}
    The moduli space $Y(E_6)$ is sch\"on.
\end{corollary}

\begin{proof}
    Immediate from \cref{thm:Y_quasilinear} and \cref{cor:qlin_schon}.
\end{proof}

From \cref{thm:Y_quasilinear}, we also obtain a description of the Chow rings of tropical compactifications of $Y(E_6)$,
including Naruki's compactification $\oY(E_6)$ and the moduli space of stable marked cubic surfaces $\wY(E_6)$.

\begin{corollary} \label{cor:Y_chow}
    Let $\oY^{\Sigma}(E_6) \subset X(\Sigma)$ be any tropical compactification of $Y(E_6)$. Then there is an isomorphism
    $A^*(\oY^{\Sigma}(E_6)) \cong A^*(X(\Sigma))$. If $\Sigma$ is unimodular, then $\oY^{\Sigma}(E_6)$ is nonsingular
    and $A^*(\oY^{\Sigma}(E_6)) \cong H^*(\oY^{\Sigma}(E_6))$.  In particular, this applies to Naruki's compactification
    $\oY(E_6)$, and to the moduli space of stable marked del Pezzo surfaces $\wY(E_6)$.
\end{corollary}

\begin{proof}
    Immediate from \cref{thm:Y_quasilinear,cor:qlin_schon,thm:qlin_chow}. (For the last sentence, note that both
    $\oY(E_6)$ and $\wY(E_6)$ are tropical compactifications by \cite{hackingStablePairTropical2009}.)
\end{proof}

\begin{remark}
    \begin{enumerate}
        \item The intersection theory of $\oY(E_6)$ has previously been studied by Colombo and van Geemen
            \cite{colomboChowGroupModuli2004}. The above result goes a step further, by allowing one to write down an
            explicit presentation for the Chow ring of $\oY(E_6)$, and indeed of any tropical compactification of $Y(E_6)$.
        \item We expect similar results to hold for moduli of marked del Pezzo surfaces of degrees 1 and 2. (Note these
            are the only remaining interesting cases, since the moduli space of marked del Pezzo surfaces of degree 4 is
            isomorphic to $M_{0,5}$, and the moduli spaces of marked del Pezzo surfaces of degrees $> 4$ are trivial.)
            Indeed, note for instance that  Hacking, Keel, and Tevelev also show that the moduli space $Y(E_7)$ of
            marked del Pezzo surfaces of degree 2 is sch\"on and the log canonical compactification $\oY(E_7)$ is a
            tropical compactification \cite{hackingStablePairTropical2009}. The results of Sekiguchi
            \cite{sekiguchiCrossRatioVarieties2000} explicitly describing the strata of $\oY(E_7)$ imply that each open
            stratum is irreducible, rational, and Chow-free.
    \end{enumerate}
\end{remark}

\bibliographystyle{amsalpha} \bibliography{ChowRingTropical}
\end{document}